\documentclass[11pt,a4paper]{article}

\usepackage{amsfonts}
\usepackage{color}
\usepackage{graphicx}
\usepackage{float}
\usepackage{alltt}

\usepackage{amsthm}
\usepackage{amsmath}
\usepackage{graphics}
\usepackage{epsfig}
\usepackage{amssymb}
\usepackage{color}

\usepackage{graphicx}
\usepackage{enumerate}
\usepackage[colorlinks=true, allcolors=blue]{hyperref}
\usepackage[normalem]{ulem}

\usepackage{amsfonts}
\usepackage{float}
\usepackage{alltt}
\usepackage{amsthm}
\usepackage{amsmath}
\usepackage{enumitem}
\usepackage{geometry}
\usepackage{amssymb}
\usepackage{ragged2e}
\usepackage{amsthm}
\usepackage{graphicx}
\usepackage[all]{xy}
\usepackage{ textcomp }
\graphicspath{ {images/} }
\usepackage{xcolor}
\usepackage{titlesec}
\usepackage{caption} 
\usepackage{subcaption} 
\usepackage{float}

\newcommand{\Eqref}[1]{(\ref{#1})}

\usepackage[normalem]{ulem}

\usepackage{geometry}

\newcommand{\N}{\mathbb{N}}
\newcommand{\R}{\mathbb{R}}
\newcommand{\C}{\mathbb{C}}

\newtheorem{defn}{Definition}
\newtheorem{prop}[defn]{Proposition}
\newtheorem{thm}[defn]{Theorem}
\newtheorem{obs}[defn]{Remark}
\newtheorem{lema}[defn]{Lemma}
\newtheorem{coro}[defn]{Corollary}

\theoremstyle{definition}
\newtheorem{ej}{Example}
\newtheorem*{dem}{Proof}
\newtheorem*{sol}{Solution}

\theoremstyle{definition} 

\newtheorem{remark}{Remark}

\title{\textbf{On the stability of IMEX BDF methods  for DDEs and PDDEs} }

\author{Ana Tercero-B\'aez$^{(a)}$
        Jesús Martín-Vaquero$^{(b)}$
\\
           $^{(a)}$   Department of Applied Mathematics, University of Salamanca, Salamanca, Spain\\
              email: id00781180@usal.es
\\
            $^{(b)}$  Department of Applied Mathematics, IUFFyM,University of Salamanca, Salamanca, Spain\\
              email: jesmarva@usal.es
}

\date{}

\begin{document}

\maketitle

\begin{abstract}
In this paper, the stability of IMEX-BDF methods for delay differential equations (DDEs) is studied based on the test equation $y'(t)=-A y(t) + B y(t-\tau)$, where $\tau$ is a constant delay, $A$ is a positive definite matrix, but $B$ might be any matrix. First, it is analyzed the case where both matrices diagonalize simultaneously, but the paper focus in the case where the matrices $A$ and $B$ are not simultaneosly diagonalizable.  The concept of field of values is used to prove a sufficient condition for unconditional stability of these methods and another condition which also guarantees their stability, but according to the step size.
 Several numerical examples in which the theory discussed here is applied to DDEs, but also parabolic problems given by partial delay differential equations with a diffusion term and a delayed term are presented.


\end{abstract}

\section{Introduction}
\label{intro}

 Many real-life phenomena require the solution of differential equations that depend partially on the past, not only the current state.
DDEs find applications in various fields, including biology, physics, chemistry, economics, and engineering. They are particularly useful for modeling systems with memory effects or systems in which time delays play a crucial role.
Examples appear in population dynamics, infectious diseases (accounting for the incubation periods), chemical and enzyme kinetics and more general control problems (see, for example, \cite{BoRi00,macdonald2008biological,MaGl77,TaMaBe00}).

 These problems are modellized by  many different types of differential equations including delay differential equations (DDEs), neutral delay differential equations (NDDEs) \cite{Jackiewicz1987/88}, integro-differential \cite{Baker06}, partial delay differential equations (PDDEs)  \cite{KotoRK,Koto-Stability},
or stochastic delay differential equations (SDDEs) \cite{baker_buckwar_2000,KUCHLER2000189}.
Solving DDEs analytically can be challenging due to the presence of time delays. Numerical methods, such as the method of steps, collocation methods, and numerical continuation, are commonly employed to approximate solutions to DDEs.

In the scientific literature, to study the stability of different numerical methods for delay differential equations, it has usually been followed the classical theory based on studying the behavior of the method when applied to the differential equation with constant delay
\begin{equation}
   \label{eq:1:EDR_EstabilidadTipica}
   y'(t) = \lambda y(t) + \gamma y(t-\tau),
\end{equation}
where $\lambda, \gamma \in \C$ (see \cite{Calvo,Stability-theta-methods-DDE,KotoRK,Koto-Stability,rihan-DDE}). However, this technique is useful to analyze the stability of numerical methods for systems of DDEs:
\begin{equation}
    y'(t)=-A y(t) + B y(t-\tau),
    \label{eq:1:EDR_EstabilidadMétodo}
\end{equation}
only when $A$ and $B$ diagonalize simultaneously.

For this reason, in this paper (Sect. \ref{stab-general-case}) we will follow a newer theory that was used to study the stability of some numerical methods for problems without delay (see \cite{MV-Pagano,LMM-Seib-Ros,LMM-Seib}), and only recently to know the behavior of the $\theta$-methods  when applied to differential equations with a constant delay, see \cite{Art-Alejandro}.

\subsection{Outline of the paper}

The outline of the paper is as follows: in \S \ref{BDF-meth} IMEX BDF methods are introduced and consistency of these methods are analyzed.  In \S \ref{Sección: Est IMEX escalar} the linear stability in a scalar DDE is studied in detail.  In this way, the linear stability of the IMEX BDF methods when both matrices diagonalize simultaneously can be easily explained in \S \ref{stab-simult-diag}.  Later, in \S \ref{stab-general-case}, the concept of field of values and its properties are provided. With this new definition, some theoretical results are given to obtain unconditional and conditional stability, and some tests with systems of DDEs are provided to better explain the results obtained.  Finally, in \S \ref{Numer-simulations}, some numerical simulations in PDDEs give a good idea of the power of the developed theory.

\section{IMEX BDF methods}\label{BDF-meth}

In general, let us consider a differential equation with delay (DDE) of the form:
\begin{equation}\label{DDE eq LMM}
    \left\{
            \begin{aligned}
              & y'(t)=f(t,y(t))+g(t,y(t),y(t-\tau)) & t\geq t_{0}\\
              & y(t) = \phi(t) & t\leq t_{0}.
            \end{aligned}
    \right.
\end{equation}

A  $s-$step linear multistep implicit-explicit (IMEX)  method (with $s\geq 1$) for solving \Eqref{DDE eq LMM} is written as:
 \begin{equation}\label{original IMEX-LMM}
        \sum_{j=0}^{s} \alpha_{j} y_{n+j}=h\left(\sum_{j=0}^{s}\beta_{j}f(t_{n+j},y_{n+j})+\sum_{j=0}^{s-1}\beta_{j}^{*} g(t_{n+j},y_{n+j}, y_{n+j-m})\right),
    \end{equation}
    $h$ being the \textit{step size} that verifies the constraint $h=\frac{\tau}{m-u}$, with $m\in\N$ and $0\leq u\leq 1$; and $\{\alpha _{j},\ \beta _{j}, \ \beta _{j}^{*}\}_{0\leq j \leq s}$ the corresponding coefficients of the method.

 Both $\alpha_{j}$ and $\beta_{j}$ will be defined by an implicit method used, since $f(t,y(t)) $ is usually a ``stiff'' term and therefore good stability properties are desired. However, it must be taken into account that the $\beta_{j}^{*}$ of the second term are defined as $\beta_{j}^{*}=\beta_{j}+\beta_{ s} \sigma_{j},$ with $0\leq j \leq s-1$, where $\sigma_{j}$ are unknown. To obtain them, let us consider the following condition extracted from \cite{Koto-Stability}, which will allow us to determine all the coefficients of the implicit-explicit multistep method:
  \begin{equation}\label{Fórmula_extrapolación_explícita}
    \sum_{j=0}^{s-1} j^{q}\cdot \sigma_{j}=k^{q}, \quad q=0,1,...,p-1,
\end{equation}
  $p$ being the desired order of the method.  Thus, we easily observe that assuming that $s=p$ in \Eqref{Fórmula_extrapolación_explícita}, the coefficients $ \sigma_{j}$ are defined univocally.

Concrete examples of this type of multistep methods are IMEX BDF2 and IMEX BDF3.  Let us impose from here on for simplicity that the parameter $u$ defined within the step size $h=\frac{\tau}{m-u}$ is null ($u=0$). The construction of these methods is developed as follows (see \cite{Koto-Stability}, for example).

Regarding the classical BDFs (BDF2 and BDF3), we will follow the procedure set out in \cite{solving-ODE-I}[\S III], developed for ordinary differential equations. First, we will calculate the coefficients using the formula:
\begin{equation*}
    \sum_{j=1}^{s}\frac{1}{j}\triangledown^{j}y_{n+1}=h f_{n+1}, \quad \text{with} \quad \triangledown^{j}y_{i}=\triangledown^{j-1}y_{i}-\triangledown^{j-1}y_{i-1}, \ j \in \N,
\end{equation*}
where $\triangledown^{0}y_{i}=y_{i}$.  As for the delayed part, we will use \Eqref{Fórmula_extrapolación_explícita}.

With all the above mentioned, the IMEX-BDF2 and IMEX-BDF3 methods applied to the general DDE \Eqref{DDE eq LMM} will have, respectively, the form
\begin{equation}\label{IMEX-BDF2 general}
    \frac{3}{2}y_{n+1}-2y_{n}+\frac{1}{2}y_{n-1}=h( f_{n+1}+2g_{n-m}-g_{n-1-m}),
\end{equation}
\begin{equation}\label{IMEX-BDF3 general}
    \frac{11}{6}y_{n+1}-3y_{n}+\frac{3}{2}y_{n-1}-\frac{1}{3}y_{n-2}=h(f_{n+1}+3g_{n-m}-3g_{n-1-m}+g_{n-2-m}).
\end{equation}

\vspace{0.1cm}
It is possible to derive higher-order IMEX BDF methods in a similar way. Actually, it is well-known that implicit BDF methods have good absolute stability properties up to sixth-order.

\subsection{  Consistency of the proposed methods  }\label{Consistency}

Convergence is normally considered as a sum of stability plus consistency. We will study in more detail in the following sections the notion of stability, applied to scalar linear equations or linear systems of equations. Turning now to the second concept, we have the following definitions (see \cite{ConsistencLMM} and references therein):
\begin{defn}\label{IMEX consistency definition}
We will state that the implicit-explicit linear multistep method \Eqref{original IMEX-LMM} is consistent if the local truncation error, defined as
$$(T_{h})_{n}=\frac{1}{h}\sum_{j=0}^{s} \alpha_{j} y(t_{n+j})-\left(\sum_{j=0}^{s}\beta_{j}f(t_{n+j},y(t_{n+j}))+\sum_{j=0}^{s-1}\beta_{j}^{*} g(t_{n+j},y(t_{n+j}), y(t_{n+j-m}))\right)$$
where $y(t_{n+i})$ represents the exact solution of the DDE \Eqref{DDE eq LMM} at time $t_{n+i}$, verifies
$$\|T_{h}\|_{\infty} \longrightarrow 0, \quad \mbox{when} \quad h\rightarrow 0.$$
\end{defn}
\begin{defn}\label{Def IMEX order}
The linear multistep method \Eqref{original IMEX-LMM} is of order $p$ if
$$\|T_{h}\|_{\infty}=O(h^{p}), \quad \mbox{when} \quad h\rightarrow 0.$$
\end{defn}

Using Taylor series in both equations in differences \Eqref{IMEX-BDF2 general} and \Eqref{IMEX-BDF3 general}, it is possible to obtain the following results:
\begin{thm}
The local truncation error (LTE) of the IMEX BDF2 method is of the form
\begin{equation*}
\begin{aligned}
LTE_{BDF2}= c_{2}h^{2}+O(h)^{3}
\end{aligned}
\end{equation*}
where $c_{2}$ is a function that depends on $f(\cdot)$, $g(\cdot)$ and some of its derivatives. A file named "1. IMEX BDF2 theory" located in folder \textit{"Codes article"} at  \href{https://github.com/anatb3/Article-Documentation.git}{https://github.com/anatb3/Article-Documentation.git} contains the Mathematica codes used to obtain the full expression. With this, it can be stated that the IMEX BDF2 is a second-order method.
\end{thm}
\begin{thm}
The LTE of the IMEX BDF3 method is $LTE_{BDF3}= c_{3}h^{3}+O(h)^{4}$ where $c_{3}$ is a function that depends on $f(\cdot)$, $g(\cdot)$ and some of its derivatives, and therefore it is a third-order method. The previous link contains also a file named "2. IMEX BDF3 theory" located again in folder "Codes article" where it is explained how to obtain the full expression of the LTE for this method.
\end{thm}

  \section{ Linear stability in a scalar DDE}\label{Sección: Est IMEX escalar}

  The simplest case for the analysis of the stability of a differential equation with delay is the scalar linear case. In the scientific literature, see \cite{Stability-theta-methods-DDE,Koto-Stability}, it is usually considered the following test problem to analyse the stability
  \begin{equation}\label{DDE eq LMM lineal}
    \left\{
        \begin{aligned}
              & y'(t)=-\lambda y(t)+\gamma y(t-\tau) & t\geq t_{0}\\
              & y(t) = \phi(t) & t\leq t_{0},
        \end{aligned}
    \right.
\end{equation}
however, we will slightly modify it into
  \begin{equation}\label{DDE eq LMM lineal2}
    \left\{
        \begin{aligned}
              & y'(t)=- ( \lambda ( y(t)+ \mu y(t-\tau) ) ) & t\geq t_{0}\\
              & y(t) = \phi(t) & t\leq t_{0},
        \end{aligned}
    \right.
\end{equation}
with $\lambda>0$ real number, $|\gamma|<\lambda$ and $\mu = - \frac{\gamma}{\lambda}$. In this way, the general IMEX method \Eqref{original IMEX-LMM} applied to the previous system of equations, defining $z=-\lambda h$ and $w=\gamma h = - z \mu$, will have the form:
  \begin{equation}\label{IMEX-LMM lineal}
    \sum_{j=0}^{s} \alpha_{j} y_{n+j}=z\sum_{j=0}^{s}\beta_{j}y_{n+j}+w\sum_{j=0}^{s-1}\beta_{j}^{*}y_{n+j-m}.
\end{equation}

The stability for the scalar linear delay equation is reduced to the analysis of the roots of the characteristic equation. In this sense, we obtain the polynomials
\begin{equation}\label{polinomios IMEX}
    \rho(\zeta)=\sum_{j=0}^{s} \alpha_{j} \zeta^{j}, \quad \sigma(\zeta)=\sum_{j=0}^{s} \beta_{j} \zeta^{j}, \quad \sigma^{*}(\zeta)=\sum_{j=0}^{s-1} \beta_{j}^{*} \zeta^{j},
\end{equation}
and therefore
\begin{equation}\label{IMEX-LMM eq caract lineal}
  P( \mu, z, \zeta) =  \zeta^{m}(\rho(\zeta)-z\sigma(\zeta))+z\mu\sigma^{*}(\zeta)=0.
\end{equation}

The following concepts are obtained from there:
\begin{defn}\label{Def estab eq charact} The characteristic equation \Eqref{IMEX-LMM eq caract lineal} is stable for $z\in \R_{<0 }\cup \{-\infty\}$ and $\mu\in\C$, if every solution satisfies $|\zeta|<1$ for that particular $z\in \R_{<0 }\cup \{-\infty\}$.
\end{defn}
\begin{defn}\label{Def unconditional stable region} The $P-$stability region of the method \Eqref{IMEX-LMM lineal}, denoted by $D_{P}$, is: $$D_{P}=\underset{m\geq 0}{\bigcap} D^{m}_{P}\quad \text{where} \quad D^{m}_{P}=\left\{( z,\mu) \in \C^{2} \ / \ \Eqref{IMEX-LMM eq caract lineal} \ \text{is stable} \right\}.$$ \end{defn}
\begin{defn} \label{Def border region est} For a fixed $m \in \N$, if  $z\in\R_{<0}\cup\{-\infty\}$ is fixed,  a stability region $D_{z} \subset \R^2$ is obtained. Its boundary will be denoted by
$$\Gamma_{z}= \left\{ \mu=  - \frac{\zeta^{m}(\rho(\zeta)-z\sigma(\zeta))}{z\sigma^{*}(\zeta)}\in\C \ / \ \zeta=e^ {i\theta}, \ 0\leq \theta \leq 2\pi \right\}.$$
\end{defn}
\begin{defn}\label{Def sigma_z}
The point of $\, \Gamma_{z}$ with minimum  distance respect to the origin will be $\sigma_{z}=\underset{\mu\in\Gamma_{z}}{\inf}|\mu|.$
\end{defn}

\begin{defn}\label{Def S_A}
Given the initial value problem
    \begin{equation} \label{Dahlquist-test}
        \left\{
        \begin{aligned}
            &y'(t)=\lambda y(t) &t\geq t_{0}\\
            &y(t_{0})=y_{0},
        \end{aligned}
        \right. \quad \text{with $\lambda \in \C$},
    \end{equation}
    if a numerical method applied to the ODE can be written as the recurrence equality  $y_{n+1}=R(h_{n+1}\lambda) y_{n}$, where $R(z)$ is usually called stability function (or $A-$stability function), \textit{ the $A-$stability region } for (\ref{Dahlquist-test}) is defined as
 \begin{equation} \label{S_A}
 S_{A}:=\{ z\in\C \ / \ |R(z)|<1 \}.
 \end{equation}

\end{defn}


In this way, it is possible to characterize the $P$-stability region as follows:
\begin{lema} \label{Lemma-Koto} \emph{\cite[Theorem 1]{Koto-Stability}.} Let us assume that the temporal coefficients of the method $(\alpha_{0},...,\alpha_{s} )$ and $(\beta_{0},...,\beta_{s})$ are linearly independent, and let us consider the following statements:
 \begin{enumerate}[label=(\alph*)]
 \item $z\in S_{A}, \ |\mu|<\sigma_{z}$
 \item $(z,\mu)\in D_{P }$
 \item $z\in S_{A}, \ |\mu|\leq\sigma_{z}$
 \end{enumerate}
 Then it turns out that a) $\Rightarrow$ b) $\Rightarrow$ c).
 \end{lema}

As a result, since $[- \infty, 0)$ is inside the A-stability region for both Implicit BDF2 and BDF3 methods, we obtain the following result:
 \begin{coro}\label{Prop D(0,sigma)cD_z}
For any $z_{j}\in\R_{<0}\cup \{-\infty\}$, $D(0,\sigma_{z_{j}})\subseteq D_{z_{j}}$ holds.
\end{coro}

 Considering \Eqref{IMEX-BDF2 general} and \Eqref{IMEX-BDF3 general}, where the coefficients necessary to define the respective characteristic equations are collected, we will have in each case that these are:
\begin{equation}\label{IMEX-BDF2 eq caract}
-\frac{1}{2} \xi ^{-m} \left(-\xi ^{2 m}+4 \xi ^{2 m+1}-3 \xi ^{2 m+2}+2 z \xi ^{2 m+2}-4 \mu \xi z+2 \mu z\right)=0
\end{equation}
\begin{equation}\label{IMEX-BDF3 eq caract}
-\frac{1}{6} \xi ^{-m} \left(2 \xi ^{2 m}-9 \xi ^{2 m+1}+18 \xi ^{2 m+2}-11 \xi ^{2 m+3}+6 z \xi ^{2 m+3}-18 \mu \xi ^2 z+18 \mu \xi z-6 \mu z\right)=0
\end{equation}

By imposing $\xi=e^{i\theta}$ on these equations and restricting the variable $z$ to the negative real line by definition, we can consider the stability regions of each of the methods in $\R \times \C \simeq \R^{3}$.

For different values of $m$ and $z<0$, we will demonstrate below, see Proposition \ref{Prop internal region D_{z}}, that the stability region $D_{z}$ coincides with the innermost region of the curve $\Gamma_{z}$. Let us assume for now that this statement is true and let, for example, $m=0$. The stability regions for this hypothetical case are shown in Fig. \ref{fig:graficas métodos IMEX BDF}.
 \begin{figure}[h]
    \centering
    \begin{subfigure}[b]{0.35\textwidth}
        \centering
        \includegraphics[width=\textwidth]{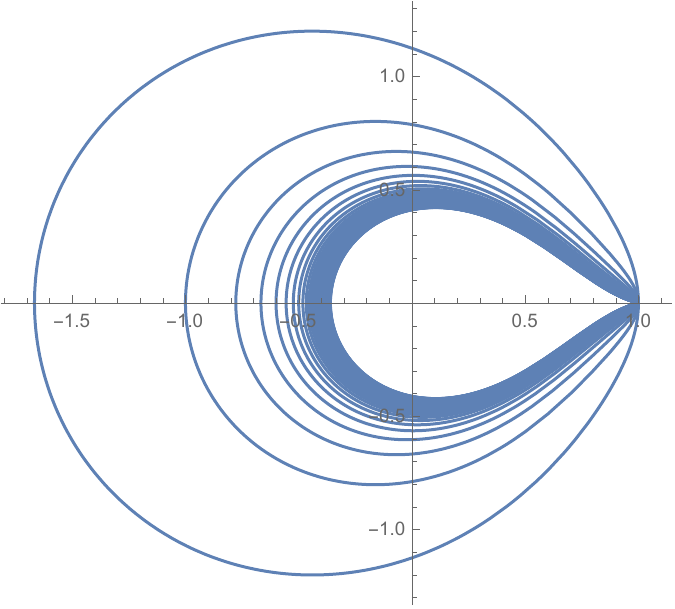}
        \caption{ IMEX BDF2 Method}
        \label{fig:figura1}
    \end{subfigure}
    \quad \quad
    \begin{subfigure}[b]{0.35\textwidth}
        \centering
        \raisebox{0.32cm}{\includegraphics[width=\textwidth]{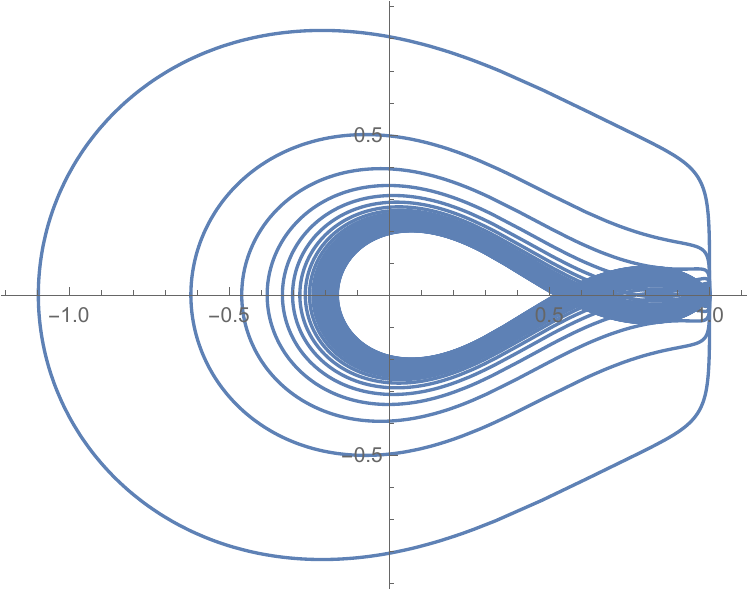}}
        \caption{IMEX BDF3 Method}
        \label{fig:figura2}
    \end{subfigure}
    \caption{  $\Gamma_{z}$ with $m=0$, $u=0$, and $z=-1, \ldots,-50$ ($\Gamma_{-50}$ is the most interior, ..., $\Gamma_{-1} $ is the most exterior region.) }
    \label{fig:graficas métodos IMEX BDF}
\end{figure}

 For $m\geq 1$, let us set a concrete value $z=-1$. The regions, which we remember will correspond to the innermost region of the represented curves, are represented in Fig. \ref{fig:graficas IMEX BDF3 z=-1}.  Additionally, we can observe how for large values $ m \gg 1$, $D_z \rightarrow D(0, \sigma_z)$ for any value $z<0$.
 \begin{figure}[h!]
    \centering
    \begin{subfigure}{0.25\textwidth}
        \centering
        \includegraphics[width=\linewidth]{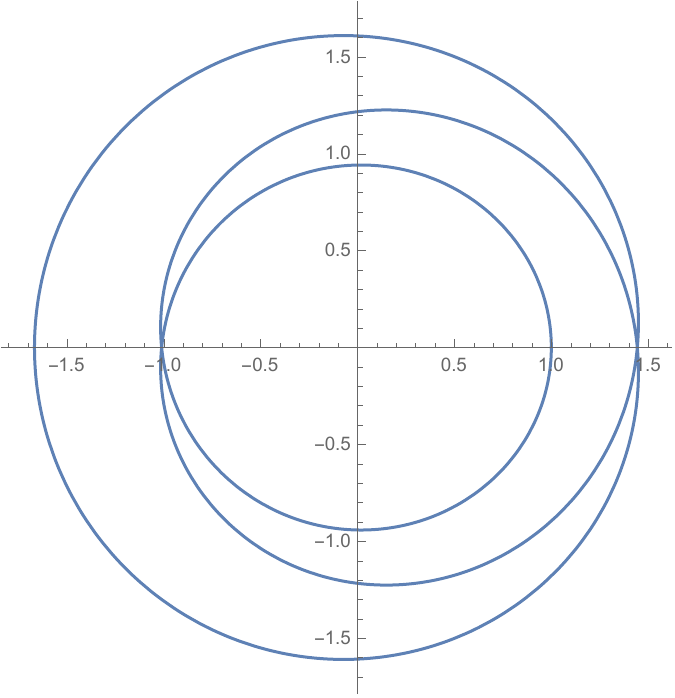}
        \caption{m=1 IMEX BDF2}
        \label{fig:figura3}
    \end{subfigure}
    \hspace{1cm}
    \begin{subfigure}{0.25\textwidth}
        \centering
        \includegraphics[width=\linewidth]{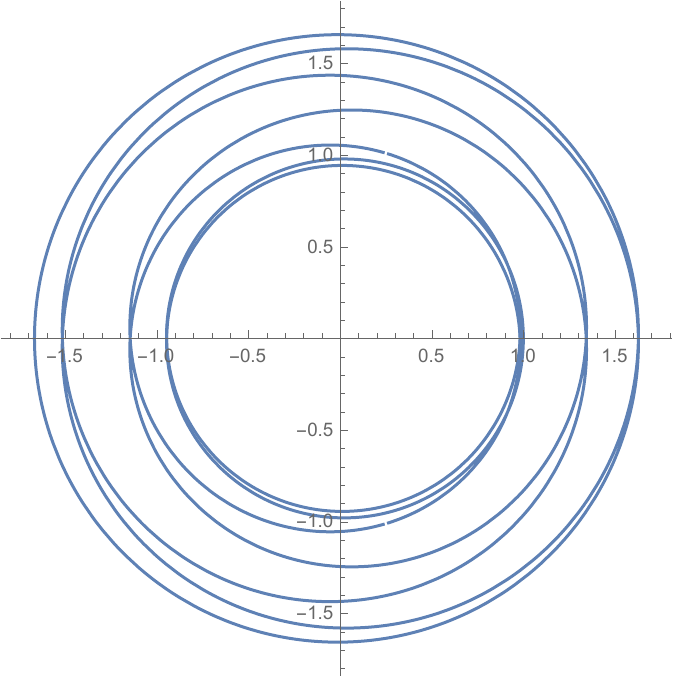}
        \caption{m=3 IMEX BDF2}
        \label{fig:figura4}
    \end{subfigure}
    \hspace{1cm}
    \begin{subfigure}{0.25\textwidth}
        \centering
        \includegraphics[width=\linewidth]{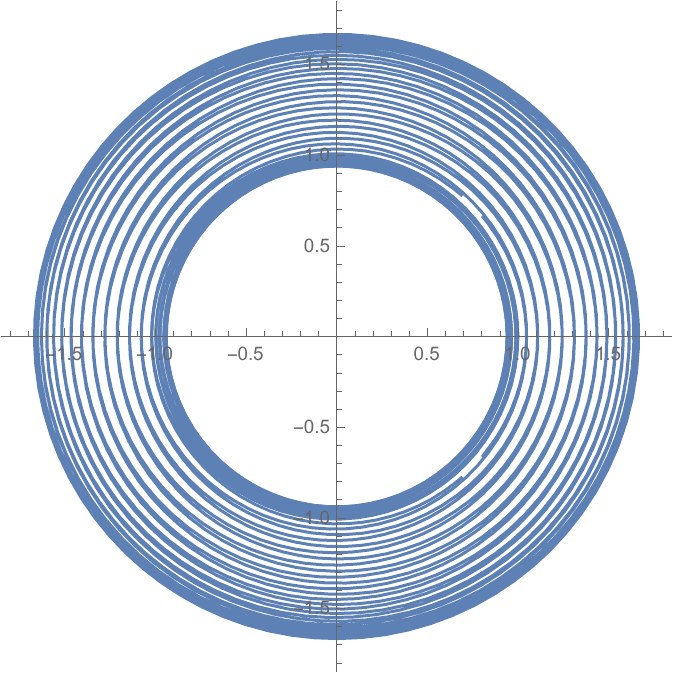}
        \caption{m=20 IMEX BDF2}
        \label{fig:figura5}
    \end{subfigure}\\
    \vspace{0.1cm}
    \begin{subfigure}{0.25\textwidth}
        \centering
        \includegraphics[width=\linewidth]{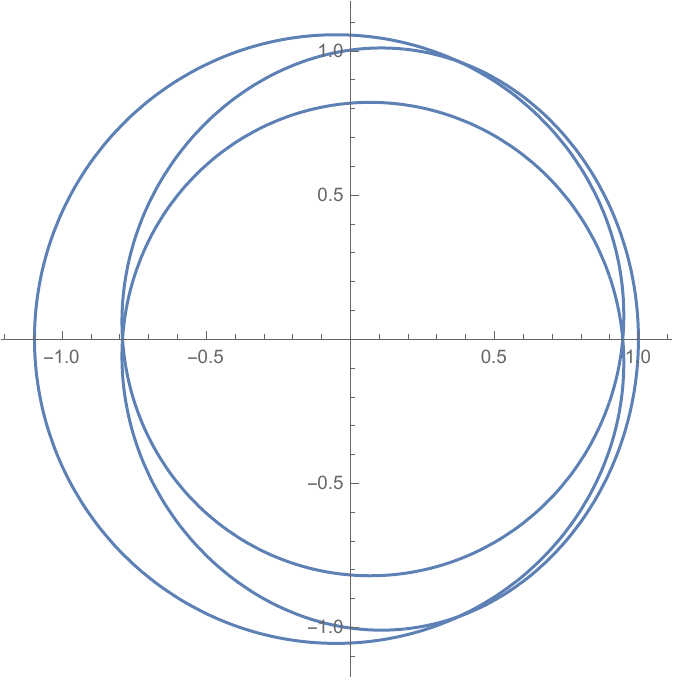}
        \caption{m=1 IMEX BDF3}
        \label{fig:figura6}
    \end{subfigure}
    \hspace{1cm}
    \begin{subfigure}{0.25\textwidth}
        \centering
        \includegraphics[width=\linewidth]{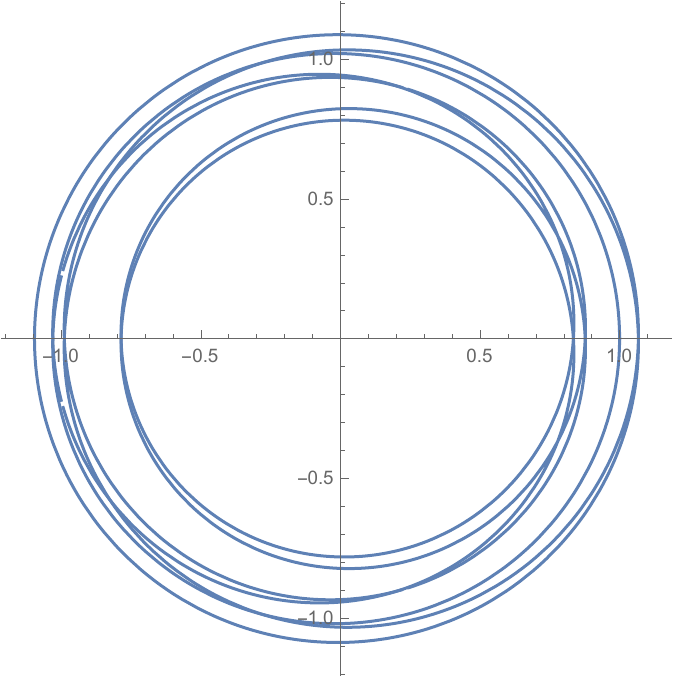}
        \caption{m=3 IMEX BDF3}
        \label{fig:figura7}
    \end{subfigure}
    \hspace{1cm}
    \begin{subfigure}{0.25\textwidth}
        \centering
        \includegraphics[width=\linewidth]{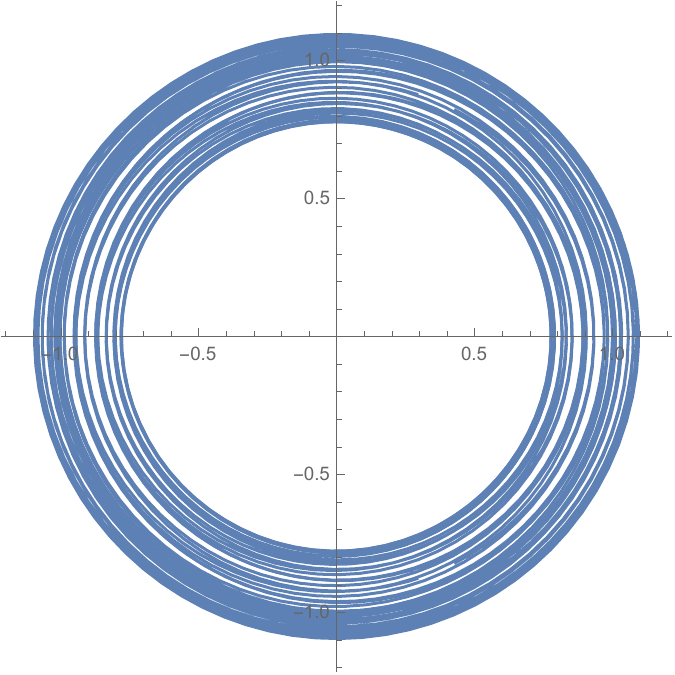}
        \caption{m=20 IMEX BDF3}
        \label{fig:figura8}
    \end{subfigure}

    \caption{ $\Gamma_{-1}$ for different values of $m$.}
    \label{fig:graficas IMEX BDF3 z=-1}
\end{figure}

\begin{obs}
For the same values of $z$ and $m$ in both procedures, the stability regions of the IMEX BDF2 method are slightly larger than those of the IMEX BDF3 method. Obviously, the larger this region, the larger the step size that can be imposed on the method without risking stability, thus decreasing the total computational time to solve the differential equation.
\end{obs}

\vspace{0.2cm}
With this we will develop the following concept, from which we can consider various implications:
\begin{defn}\label{Def region est incond D}
We will call the unconditional stability region $D$ to the set of values $\mu\in\C$ such that \Eqref{IMEX-LMM eq caract lineal} is stable for all $z\in\R_{<0}\cup\{-\infty\}$ and all $m\geq 0$. This is
$$D_{z}=\left\{\mu\in\C \ / \ \Eqref{IMEX-LMM eq caract lineal} \ \text{is stable for a} \ z\in\R_{<0}\cup\{-\infty\} \ \text{fixed and }\ \forall \ m\geq 0\right\},$$ $$D=\underset{z \in \R_{<0} \cup \{-\infty \}}{\bigcap} D_{z}.$$
 \end{defn}
\begin{prop}\label{Prop internal region D_{z}} Let $\Tilde{D}_{z}$ be the innermost region of the curve $\Gamma_{z}$. It is then verified that $\Tilde{D}_{z}=D_{z}$.
\end{prop}
\begin{dem}

Let us first see that the mentioned region $\Tilde{D}_{z}$ exists, that is, the curve $\Gamma_{z}$ is closed and bounded. Knowing that $e^{i\theta}=\cos(\theta)+i\sin(\theta)$, we will then have from the Definition \ref{Def border region est} that any point on said curve is of the form
$$\mu=-\frac{(\cos(\theta m)+i\sin(\theta m))\sum_{j=1}^{s}(\alpha_{j}-z\beta_{j})(\cos(\theta j)+i\sin(\theta j))}{z\sum_{j=1}^{s-1}\beta_{j}^{*}(\cos(\theta j)+i\sin(\theta j))}$$

Multiplying by $\sum_{j=1}^{s-1}(\cos(\theta j)-i\sin(\theta j))$ to the previous expression, both the numerator and the denominator, and taking into account the trigonometric identities, we obtain the expressions for the real and imaginary parts of $\mu$ such that
$$\Re(\mu)=-\frac{\sum_{j=1}^{s}\alpha_{j}-z\beta_{j}}{z\sum_{j=1}^{s-1}\beta_{j}^{*}}\cos(\theta(m+s)), \quad \quad \quad
\Im(\mu)=-\frac{\sum_{j=1}^{s}\alpha_{j}-z\beta_{j}}{z\sum_{j=1}^{s-1}\beta_{j}^{*}}\sin(\theta(m+s))\cdot i.$$
Both functions are continuous in their domains, and symmetric with respect to the real axis (due to the even symmetry of the cosine function and the odd symmetry of the sine function), and for $\theta=\{0, \pi\}$, it is true that $\Im(\mu)=0$. Therefore, it is clear that the curve $\Gamma_{z}$ is closed and bounded.

\vspace{0.1cm}
As to demonstrate that $\Tilde{D}_{z}=D_{z}$, it is possible to do it in a similar way as in \cite[p.8]{Art-Alejandro}: it is clear that for any $z \in \R_{<0}$,   $0 \in D_z$ and, by continuity, $\Tilde{D}_z \subseteq D_z$.    As for the opposite inclusion it is possible to use Cauchy's integral theorem to demonstrate that only in the innermost region of the curve $\Gamma_{z}$, $P( \mu, z, \zeta)$ has all the roots with $|\zeta | < 1$. $\hfill\square$
\end{dem}

In order to study  the stability for different values of $z$ of both methods, from Lemma \ref{Lemma-Koto}, we observe that  it is necessary to first calculate the values for $\sigma_{z}$:

  \subsection{ $\sigma_{z}$ for IMEX BDF2 and IMEX BDF3 methods }\label{Section:sigmaIMEXBDF}

Let  (\ref{IMEX-BDF2 eq caract})-(\ref{IMEX-BDF3 eq caract}) be the characteristic equations of the IMEX BDF2 and IMEX BDF3 methods, respectively. Imposing $\xi=e^{i\theta}$, with $\theta\in [0,2\pi]$, on these and solving with respect to the parameter $\mu$, we arrive at the equations of the curves $\Gamma_{z}$ for each $z<0$ and each method:
$$\mu_{m, z, \theta}=\frac{(e^{i\theta})^{2m}(-1+4e^{i\theta}-3e^{2i\theta}+2e^{2i\theta}z)}{2(-1+e^{i\theta})z} \quad \quad \text{IMEX BDF2}$$

$$\mu_{m, z, \theta}=\frac{(e^{i\theta})^{2m}(2-9e^{i\theta}+18e^{2i\theta}-11e^{3i\theta}+6e^{3i\theta} z)}{6(1-3e^{i\theta}+3e^{2itheta})z} \quad \quad \text{IMEX BDF3}$$

Let us suppose that for a given $z$, $\sigma_{z}$ is obtained from $\sigma_{z} = |\mu_{m, z, \theta}|$ for some values of $m$ and $ \theta$.  It is clear for both methods that $|\mu_{m, z, \theta}|=|\mu_{0, z, \theta}|$, which implies that the study can be reduced to the hypothetical case $m=0$. The calculations for the first method (those for the second being analogous) are as follows:
\begin{equation*}
\begin{aligned}
\left|\mu_{m, z, \theta}\right| &=\left|\frac{(e^{i\theta})^{2m}(-1+4e^{i\theta}-3e^{2i\theta}+2e^{2i\theta}z)}{2(-1+e^{i\theta})z}\right|=\frac{\left|e^{i\theta}\right|^{2m}\cdot\left|-1+4e^{i\theta}-3e^{ 2i\theta}+2e^{2i\theta}z\right|}{\left|2(-1+e^{i\theta})z\right|}\underset{\underset{\left|e^{i\theta}\right|=1}{\uparrow}}{=}\\
 &=\frac{\left|-1+4e^{i\theta}-3e^{2i\theta}+2e^{2i\theta}z)\right|}{\left|2(-1+e^{i\theta})z\right|}=\left|\mu_{0, z, \theta}\right|
 \end{aligned} \end{equation*}

\vspace{\baselineskip}
\underline{IMEX BDF2}

 Let us see how the aforementioned minimum function is constructed. Starting from the curve equations above, we have that the stability regions in this specific case are delimited by the formula:
$$\mu_{BDF2}=\frac{-1+4e^{i\theta}-3e^{2i\theta}+2e^{2i\theta}z}{2(-1+e^{i\theta})z}, \quad \theta\in [-\pi ,\pi].$$

For any $z<0$ and $\theta\in [0,\pi]$
(we can reduce the study of $\min |\mu_{0, z, \theta}| $ to this interval because of symmetries), we have a specific point in $\Gamma_{z}$, and its distance to the origin can be calculated by the following
 $$ \phi(z,\theta) := \left|\mu_{BDF2}(z,\theta)\right|^2 = \frac{13-6z+2z^{2}+8(z-2)\cos{\theta}+(3-2z)\cos{2\theta}}{2z^{2}(5-4\cos{\theta})}.$$
\\
Consequently, for a given and fixed $z$ value, we can determine the minimum values of $\phi(z, \cdot)$ by calculating when
$$  \frac{\partial \phi(z,\theta)}{ \partial \theta} = 0.$$
For  $\theta\in [0 ,\pi]$ there are three possible minimum values:
$$\theta=0, \quad \theta=2\arctan{\sqrt{\frac{4-2z-2z^{2}-\sqrt{-15+4z+52z^{2}-32z^{3}}}{-31+20z+2z^{2}}}} \quad \text{or} \quad \theta=\pi.$$

\vspace{0.2cm}
Applying each of these $\theta$ values to $\phi(z,\theta)$, we arrive at three functions of $z$ that, depending on said parameter, take values some greater than others. By performing comparisons to find out which of them then produces the minimum point $\sigma_{z}$ for each $z$, we finally obtain the function $\psi(z):=\sigma_{z}$:
\begin{equation}\label{Fun minimos IMEX BDF2}
\psi(z)= \left\{
\begin{aligned}
 \psi_1(z)=1,  & \quad \quad z\in \left[-\frac{1}{\sqrt{2}},0\right),\\
  \psi_2(z)=-\scalebox{1.2}{$\frac{\sqrt{1+2z+\sqrt{-(-3+2z)(1+2z)(-5+8z)}}}{2\sqrt{2}z}$},  &  \quad  \quad \scalebox{1}{$z\in \left[\frac{1}{2}(-10-9\sqrt{2}),-\frac{1}{\sqrt{2}}\right)$},\\
  \psi_3(z)=\frac{-4+z}{3z}, &  \quad  \quad   z\in \left(-\infty,\frac{1}{2}(-10-9\sqrt{2})\right).
\end{aligned}
\right.
\end{equation}
Evaluating the functions at the points $-\frac{1}{\sqrt{2}}$, $\frac{1}{2}(-10-9\sqrt{2})$, it is possible to check that $\psi(z)$ is a continuous function, so the extremes can be taken arbitrarily within some intervals or others. The graphical representation of this function can be found in  Fig. \ref{fig:graficas psi(z)}.

With all of that we can conclude the following:
\begin{thm}\label{Non-decreasing psi}
$\psi(z)$  is non-decreasing for every value $z$ belonging to the interval $[-\infty, 0)$.
\end{thm}

\begin{dem}

 Since it is defined piecewise and continuous, it will be enough to check this in each of intervals where it is divided. In terms of its derivatives this is equivalent to checking that,  $\psi'_1(z), \psi'_2(z), \psi'_3(z) \geq 0$ in their own intervals. Numerically, it is easy to observe that $\psi'_1(z)= 0$, while $ \psi'_2(z), \psi'_3(z) > 0$. The graphical representation of the derivatives of $\psi'(z)$ and $\Tilde{\psi}'(z)$ can be found in Fig. \ref{fig:graficas der psi(z)}. $\hfill\square$
\end{dem}

\begin{figure}[h]
    \centering
    \begin{subfigure}[b]{0.35\textwidth}
        \centering
        \includegraphics[width=\textwidth]{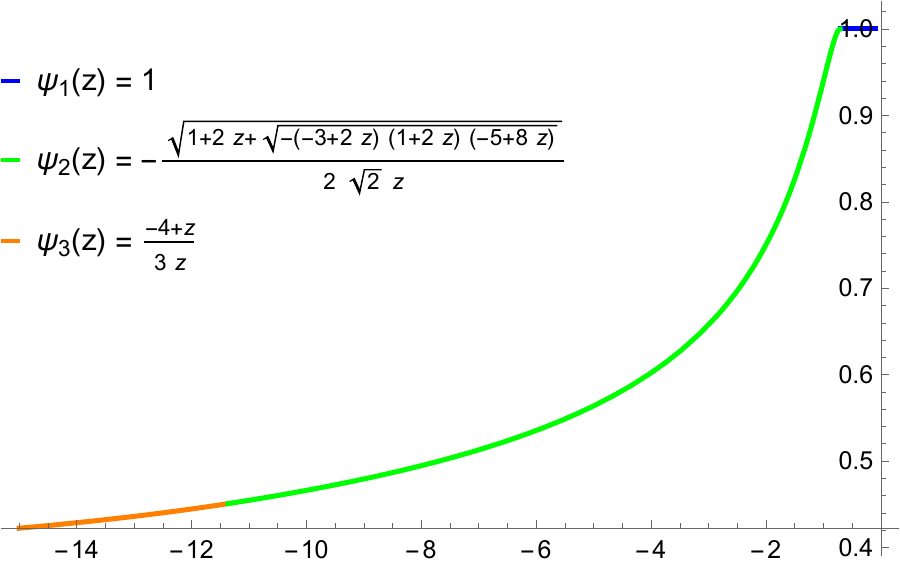}
        \caption{$\psi(z)$}
    \end{subfigure}
    \quad \quad
    \begin{subfigure}[b]{0.35\textwidth}
        \centering
        \includegraphics[width=\textwidth]{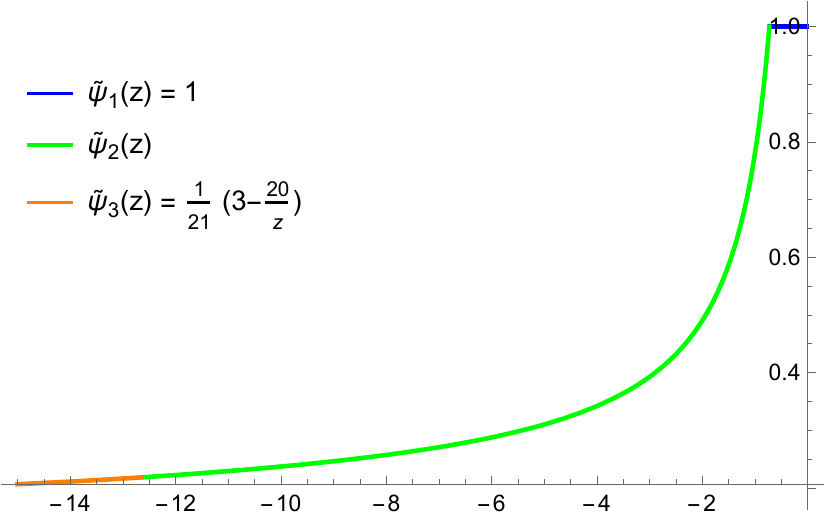}
        \caption{$\Tilde{\psi}(z)$}
    \end{subfigure}
    \caption{  Representation of $\psi(z)$ and $\Tilde{\psi}(z)$ functions. }
    \label{fig:graficas psi(z)}
\end{figure}
\begin{figure}[h]
    \centering
    \begin{subfigure}[b]{0.35\textwidth}
        \centering
        \includegraphics[width=\textwidth]{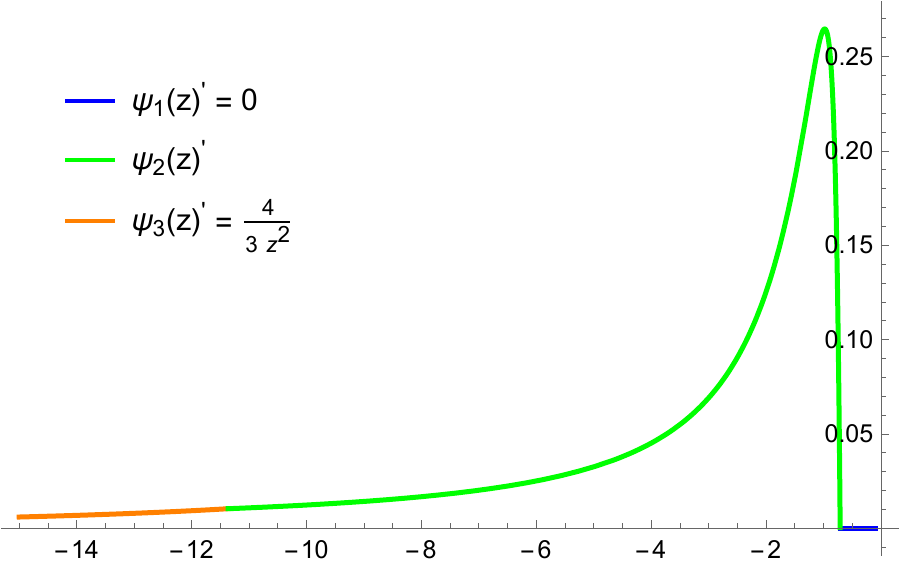}
        \caption{$\psi'(z)$}
    \end{subfigure}
    \quad \quad
    \begin{subfigure}[b]{0.35\textwidth}
        \centering
        \includegraphics[width=\textwidth]{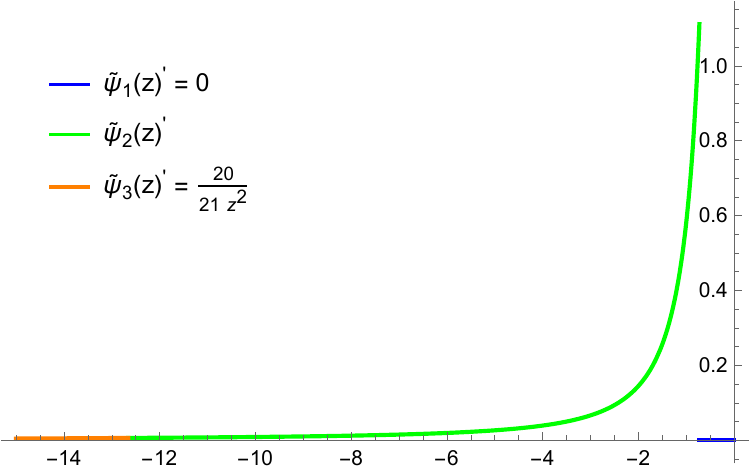}
        \caption{$\Tilde{\psi}'(z)$}
    \end{subfigure}
    \caption{  Representation of $\psi'(z)$ and $\Tilde{\psi}'(z)$ functions. }
    \label{fig:graficas der psi(z)}
\end{figure}

\vspace{\baselineskip}
\underline{IMEX BDF3}

 Reasoning in a similar way to the previous method, we arrive at the fact that the minimum function starts, in this case, from the equation of the curve
$$\mu_{BDF3}=\frac{2-9e^{i\theta}+18e^{2i\theta}-11e^{3i\theta}+6e^{3i\theta} z}{6(1-3e^{i\theta}+3e^{2i \theta})z}.$$
 $ \tilde{ \phi }(z,\theta)  := |\mu_{BDF3}(z,\theta)|^2 $ is calculated as previously, and also the points, $\theta\in [0 ,\pi]$, where $$\frac{\partial  \tilde{ \phi}(z,\theta)   }{ \partial \theta} = 0.$$

Thus, the function $ \tilde{ \psi(z) }:=\sigma_{z}$   has the following form:
\begin{equation}\label{Fun minimos IMEX BDF3}
\tilde{ \psi }(z) =
\left\{
\begin{aligned}
    \tilde{ \psi }_1(z) =1, & \quad \quad z\in \left[-0.722965,0 \right),\\
    \tilde{ \psi }_2(z), &  \quad  \quad z\in \left[-12.655874,-0.722965\right),\\
    \tilde{ \psi }_3(z) =\frac{1}{21}\left(3-\frac{20}{z}\right),  & \quad \quad z\in \left(-\infty,-12.655874\right).
\end{aligned}
\right.
\end{equation}
 Due to the complexity of the function $\tilde{ \psi}_2(z) $, we will not provide the full expression, but it is explained how to obtain it in file "2. IMEX BDF3 theory" at \href{https://github.com/anatb3/Article-Documentation.git}{https://github.com/anatb3/Article-Documentation.git}.  Again, the evaluation at the points $-0.722965$ and $-12.655874$ returns that $\tilde{ \psi }(z)$ is a continuous function, and studying the derivatives of $\tilde{ \psi}_i(z) $ lead us to the following result:

\begin{thm}\label{Non-decreasing minimum fun}
The function $\tilde{ \psi}(z) $ is non-decreasing for every value of $z<0$.
\end{thm}

\section{Linear stability of IMEX-BDF methods in case of simultaneous diagonalizability}\label{stab-simult-diag}

In this section we prove some sufficient conditions for the stability of IMEX BDF methods assuming that matrices $A$ and $B$ in equation \Eqref{eq:1:EDR_EstabilidadMétodo} are simultaneously diagonalizable. In order not to make the statements of propositions and theorems redundant, we will avoid repeating the validity of this assumption every time. Let us first introduce the set of the so-called generalized eigenvalues of the splitting $(A,B)$  \cite{MV-Pagano,LMM-Seib-Ros}:
\begin{equation*}\label{gen_eig_set}
 \bar{ \mu } (A,B):=\left\{\mu \in \mathbb{C}: \mu A v = B v, v \neq {\bf{0}}\right\}.
\end{equation*}
Note that the elements of $ \bar{\mu }(A,B)$ (when matrices  diagonalize simultaneously) are just the eigenvalues of $A^{-1}B$. Furthermore, from now on we use $\sigma(X)$ to denote the eigenvalues of a matrix $X$.

On another hand, we are also deriving two functions: $\chi(r)$ and $  \tilde{ \chi }(r) $, such that  $( \chi \circ  \psi )(z) = z$  and $( \tilde{ \chi } \circ \tilde{ \psi} )(z) = z$, respectively in the intervals $\left( - \infty,  -\frac{1}{\sqrt{2}} \right)$ and $( - \infty,  -0.722965)$.

\vspace{0.2cm}

\underline{Function $\chi(r)$, for the stability of IMEX BDF2}

\vspace{0.12cm}
To study the stability of IMEX BDF2, let us consider the function $\psi(z)$ given by equation (\ref{Fun minimos IMEX BDF2}), which by the Theorem \ref{Non-decreasing psi} we know is nondecreasing.  Additionally,  taking the limit when $z$ tends to minus infinity of the minimum function $\psi(z)$, specifically looking at $\psi_3(z)$ when $z \rightarrow -\infty$, we arrive at the minimum value that is reached in the curve of the stability region, that is $\frac{1}{3}$.

 Therefore, for a given $\frac{1}{3}< r \leq 1$, let us look for the minimum $z_m$ such that $\psi(z_m)=\sigma_{z_m}=r$.  Thus, we are looking for $z_{m}=\min_{z\in\R^{-}} \psi^{-1}(r)$.  Since $\psi$ is piecewise function, we need to divide its study in several intervals.  Simple calculations provide us the new function
\begin{equation}\label{chi-IMEX-BDF2}
\chi(r)= \left\{
\begin{aligned}
    -\frac{1}{\sqrt{2}} & \quad \quad \quad \quad r=1 \\
    \psi_2^{-1}(r)    & \quad \quad \frac{3}{31}(-1+4\sqrt{2}) < r< 1\\ 
    \frac{4}{1-3r} & \quad \quad  \frac{1}{3}< r \leq \frac{3}{31}(-1+4\sqrt{2}).
\end{aligned}
\right.
\end{equation}
As for the second interval, first let us remember that $\psi_2'(z) > 0$, and therefore $\psi_2^{-1}(r)$ is well defined and can be calculated, for example, with a Newton-Raphson algorithm solving $\psi(z_d)=r$.

\vspace{0.2cm}

\underline{Function $ \tilde{\chi}(r)$, for the stability of IMEX BDF3}

\vspace{0.1cm}
Similarly, by  Theorem \ref{Non-decreasing minimum fun}  we know that  $\tilde{\psi}(z)$  is nondecreasing.  Additionally,  taking the limit when $z$ tends to minus infinity of the minimum function $\tilde{\psi}$, simple calculations provide us:
$$ lim_{z \rightarrow -\infty }   \tilde{\psi}_3(z) =  \frac{1}{7} $$.

Hence, let us define
\begin{equation}\label{chi-IMEX-BDF3}
\tilde{ \chi}(r)= \left\{
\begin{aligned} -0.722965  &  \quad \quad  \quad \quad r=1 \\
 \tilde{\psi}_2^{-1}(r) &   \quad \quad   0.218109 <  r<1
\\
    - \frac{20}{3(7r-1)} &   \quad \quad \frac{1}{7}< r \leq 0.218109.
\end{aligned}
\right.
\end{equation}
Again $\tilde{\psi}_2^{-1}(r)$ is well defined and can be calculated with a Newton-Raphson method.

\vspace{0.2cm}
Now, we are ready to demonstrate the following theoretical results:

\subsection{ Sufficient conditions for linear stability: case matrices  diagonalize simultaneously }

The next proposition provides a sufficient condition for the stability of IMEX BDF methods when the one-to-one correspondence between the eigenvalues of $A$ and $B$ is known.

\begin{prop}\label{prop1}
 	Let $A v_i = \lambda_i v_i$, $B v_i = \gamma_i v_i$, with $i \in \{1,\ldots,d\}$, and $v_i \in \mathbb{C}^d$. Furthermore, let $\mu_i= \lambda_i^{-1} \gamma_i$. Given a value of the time step-size $h>0$, the IMEX BDF method is stable  if   $\mu_i \in D(0, \sigma_{z_i} ), \ \forall i=1,\ldots,d$ , where $z_i= -h \lambda_i$.

    Or equivalently, it is sufficient that
    $ h< \min_i \frac{  | \chi( | \mu_i |  ) |}{ \lambda_i}  $ for IMEX BDF 2, and
    $ h< \min_i \frac{  | \tilde{ \chi} ( | \mu_i |  ) |}{ \lambda_i}  $ for IMEX BDF 3.  If $|\mu_i  | \leq \frac{1}{3}, \  \forall i=1,\ldots,d$, then IMEX BDF 2 is unconditionally stable for the problem, and if additionally $|\mu_i  | \leq \frac{1}{7}, \  \forall i=1,\ldots,d$, then IMEX BDF 3 is unconditionally stable.
\end{prop}

\begin{proof}
    If  $A$ and $B$ diagonalize simultaneously, then it holds that
    \begin{equation*}
		A=Q \Lambda Q^{-1}, \quad \text{with} \quad \Lambda=\text{diag}(\lambda_i), \qquad B = Q \Gamma Q^{-1}, \quad \text{with} \quad \Gamma=\text{diag}(\gamma_i).
    \end{equation*}
    Setting $y=Q u $, the considered test problem \Eqref{eq:1:EDR_EstabilidadMétodo} turns into
    \begin{equation*}
	u'(t)=\Lambda u(t) +  \Gamma u(t-\tau).
    \end{equation*}
    We can then write it as a system of decoupled scalar equations of the form
    $u_i=\lambda_i u_i+\gamma_i u_i (t-\tau), \ i=1,\ldots,d$.  Therefore, it is sufficient and  necessary that  $\mu_i \in D_{z_i}, \ \forall i=1,\ldots,d$.  Since $D(0, \sigma_{z_i} ) \subseteq D_{z_i} $  the first part of the statement is obtained.

    As a consequence it is necessary that  $ | \mu_i | < \sigma_{z_i} = \sigma( - h \lambda_i ), \forall i=1,\ldots,d $.    The second part of the theorem is obtained by solving this formula with respect to $h$, and the fact that $\psi(z_i) > \frac{1}{3}$ and $\tilde{ \psi}(z_i)> \frac{1}{7} $ whenever $z_i < 0 $.
\end{proof}

\begin{lema}\label{Observation}
    Taking into account the previous theorem, if we take for the IMEX BDF2 and IMEX BDF3 methods values  $z_{2}<z_{1}<0$, then $\sigma_{z_{2}}\leq \sigma_{z_{1}}$. With the latter, we then arrive at that, in the mentioned procedures, $D(0,\sigma_{z_{2}})\subseteq D(0,\sigma_{z_{1}})$ is fulfilled for all $z_{2}<z_{1}<0$.

    \begin{proof}
        It is directly obtained from Theorems  \ref{Non-decreasing psi} and \ref{Non-decreasing minimum fun}, where it was demonstrated that functions $\psi$ and $\tilde{\psi}$ are non-decreasing.
    \end{proof}
\end{lema}

As a consequence of Lemma \ref{Observation} and Proposition \ref{prop1}, we can obtain:

\begin{thm} \label{Sufficient-cond-stab-simultaneous}
Let $A v_i = \lambda_i v_i$, $0< \lambda_1 < \ldots <  \lambda_d$,  $B v_i = \gamma_i v_i$, with $i \in \{1,\ldots,d\}$, and $v \in \mathbb{C}^d$. Furthermore, let $\mu_i= \lambda_i^{-1} \gamma_i$. Given a value of the time step-size $h>0$, the IMEX BDF method is stable  if $\mu_i \in D(0, \sigma_{z_d} ), \ \forall i=1,\ldots,d$, where $z_d= -h \lambda_d$ .

 Or equivalently it is sufficient that
 $ h<  \min_i \frac{  | \chi( | \mu_i) | |}{ \lambda_d}  $ for IMEX BDF 2, and
   $ h< \min_i \frac{  | \tilde{ \chi} ( | \mu_i | ) |}{ \lambda_d}  $ for IMEX BDF 3.  Obviously,  if $|\mu_i  | \leq \frac{1}{3}, \  \forall i=1,\ldots,d$, then IMEX BDF 2 is unconditionally stable for the problem, and if additionally $|\mu_i  | \leq \frac{1}{7}, \  \forall i=1,\ldots,d$, then IMEX BDF 3 is unconditionally stable.
\end{thm}

\subsection{ Example: linear system where matrices simultaneously diagonalize }

\begin{ej}
\textit{Stability restrictions of  IMEX BDF2 and IMEX BDF3 methods applied to the linear system:}
\begin{equation*}
\left\{
\begin{aligned}
& y'(t)=-A\ y(t)+B\ y(t-1)+\tilde{f}(t) & t\geq 0\\
& y(t) = \begin{pmatrix}
e^{-t} \\
\sin(t) \\
2t^{2} \\
1+t
\end{pmatrix}
& t\leq 0
\end{aligned}
\right.
\end{equation*}
\textit{where the matrices are defined as}
$$
A=\begin{pmatrix}
39 & -27 & -9 & 5 \\
9 & 3 & -9 & 5 \\
22 & -27 & 8 & 5 \\
9 & 0 & -9 & 8
\end{pmatrix},
\quad \quad
B=\begin{pmatrix}
8 & -2 & -4 & 5 \\
4 & 2 & -4 & 5 \\
-3 & -2 & 7 & 5 \\
4 & 0 & -4 & 7
\end{pmatrix}.
$$
\textit{and $\Tilde{f}$ is such that $y(t) =(e^{-t}, \sin(t), 2t^{2}, 1+t)^{T}$ is the solution of the differential equation with delay.}
\end{ej}

\begin{sol}
Let us first calculate the definition of the function $\Tilde{f}(t)$ that appears in the delay differential equation of the problem. Knowing that $\tilde{f}(t)=y'(t)+Ay(t)-By(t-1)$ and the real solution of the DDE, we have that in this particular example:
\begin{equation*}
\begin{aligned}
\Tilde{f}(t)=& \left(e^{-t} \left(-e^t (2 t (5 t+8)+2 \sin (1-t)+27 \sin (t)-13)-8 e+38\right)\right., \\ & e^{-t} \left(e^t (-2 t (5 t+8)+2 \sin (1-t)+3 \sin (t)+\cos (t)+13)-4 e+9\right), \\ & (22+3 e) e^{-t}+2 t (t+16)-2 \sin (1-t)-27 \sin (t)-9, \\ & \left.e^{-t} \left(e^t (17-5 t (2 t+3))-4 e+9\right) \right).
\end{aligned}
\end{equation*}

Thus, the IMEX BDF2 and IMEX BDF3 methods, generally expressed by the equations (\ref{IMEX-BDF2 general})-(\ref{IMEX-BDF3 general}), when applied to this linear system, will result in
\begin{equation}\label{IMEX BDF2 ejemplos}
\frac{3}{2}y_{n+1}-2y_{n}+\frac{1}{2}y_{n-1}=h(A y_{n+1}+\Tilde{f}_{n+1}+B(2y_{n-m}-y_{n-1-m})),
\end{equation}
\begin{equation}\label{IMEX BDF3 ejemplos}
\frac{11}{6}y_{n+1}-3y_{n}+\frac{3}{2}y_{n-1}-\frac{1}{3}y_{n-2}=h(Ay_{n+1}+\Tilde{f}_{n+1}+B(3y_{n-m}-3y_{n-1-m}+y_{n-2-m})).
\end{equation}

\vspace{\baselineskip}
Now, let us check, in order to use our theory, whether the matrices verify being simultaneously diagonalizable or not. By definition we know that two matrices are simultaneously diagonalizable if there exists a diagonal matrix, $P$, such that $PAP^{-1}$ and $PBP^{-1}$ are diagonal matrices. We also know that two matrices diagonalize simultaneously if, and only if, they commute. Taking into account that, in this problem, $AB=BA$, it is possible to confirm that $A$ and $B$ commute and thus are simultaneously diagonalizable, and consequently Proposition \ref{prop1} can be used.

In this sense, denoting $\mu_{i}=\lambda_{i}^{-1}\gamma_{i}$ for every $i=1,2,3,4$ and considering the equations (\ref{chi-IMEX-BDF2}) and (\ref{chi-IMEX-BDF3}) for the methods IMEX BDF2 and IMEX BDF3, respectively, we get the step size restrictions for this delay differential equation. These are $h<h^{*}=0.157609$ for the first method mentioned and $h<h^{*}=0.0760254$ for the second, and henceforth it is possible to calculate the errors made when solving the DDE. Table \ref{tab:Autonomous error table 1} shows these errors, committed at time $t_{e}=500$, for different step sizes. Calculations for this exercise can be found in file "3. Example 1" (and in the respective files for the rest of the exercises) located in folder "Codes article" at \href{https://github.com/anatb3/Article-Documentation.git}{https://github.com/anatb3/Article-Documentation.git}.

\begin{table}[H]
\centering
\resizebox{\textwidth}{!}{%
\begin{tabular}{| c | c | c |} \hline & IMEX BDF2 & IMEX BDF3\\ \hline
$ h=0.5 $ & $ (4.8793\cdot 10^{44}, 4.8793\cdot 10^{44}, 4.8793\cdot 10^{44},4.8793\cdot 10^{44}) $ & $(4.5553\cdot 10^{119}, 4.5553\cdot 10^{119}, 4.5718\cdot 10^{119},4.5553\cdot 10^{119}  )$
\\ $h=0.25$ & $(1.3483\cdot 10^{21}, 1.3483\cdot 10^{21}, 1.3483\cdot 10^{21},1.3483\cdot 10^{21}  )$ & $(2.0328\cdot 10^{98}, 2.0328\cdot 10^{98}, 8.4992\cdot 10^{83},2.0328\cdot 10^{98} )$
 \\ $h=0.1$ & $(2.0087\cdot 10^{-1}, 1.9977\cdot 10^{-1}, 2.7420\cdot 10^{-1},2.0660\cdot 10^{-1})$ & $(5.0025\cdot 10^{22}, 5.0025\cdot 10^{22}, 4.4540 \cdot 10^{7},5.0025\cdot 10^{22})$
 \\ $h=0.05$ & $ ( 5.0196\cdot 10^{-2},4.9934\cdot 10^{2}, 6.8529\cdot 10^{-2}, 5.1666\cdot 10^{-2})$ & $ (1.5604 \cdot 10^{-5}, 2.2534\cdot 10^{-6}, 1.5605 \cdot 10^{-5}, 4.5111\cdot 10^{-10})$
 \\ $h=0.025$ & $ (1.2547\cdot 10^{-2}, 1.2483\cdot 10^{-2}, 1.7130\cdot 10^{-2}, 1.2917\cdot 10^{-2} )$ & $ (1.6232 \cdot 10^{-6}, 1.0514 \cdot 10^{-7}, 1.6233\cdot 10^{-6}, 5.0204\cdot 10^{-10} )$
 \\ $h=0.01$ & $(2.0070 \cdot 10^{-3}, 1.9973 \cdot 10^{-3}, 2.7413 \cdot 10^{-3}, 2.0666\cdot 10^{-3} )$ & $ (9.6173 \cdot 10^{-8}, 1.6631 \cdot 10^{-8}, 9.7265\cdot 10^{-8}, 4.3110\cdot 10^{-9} )$
 \\ $h=0.005$ & $ (5.0182 \cdot 10^{-4}, 4.9932 \cdot 10^{-4}, 6.8591 \cdot 10^{-4}, 5.1666\cdot 10^{-4} )$ & $ ( 5.3611 \cdot 10^{-9}, 1.9544 \cdot 10^{-8}, 8.6147 \cdot 10^{-9}, 1.6254\cdot 10^{-8}  )$
\\ \hline
Conv. Rate &  2.00005  & 2.90874
\\ \hline
\end{tabular}
}
\caption{Errors committed at time $t_{e}=500$ for $y(t) =(e^{-t}, \sin(t), 2t^{2}, 1+t)^{T}$.}
\label{tab:Autonomous error table 1}
\end{table}

The numerical results obtained using both methods approximate very well the solution of the delay differential equation. This can be easily seen in the decay of the errors when reducing the step size and also by the similarity between the estimated convergence rates and the actual rates of each method, which we know are $2$ and $3$, for the IMEX BDF2 and IMEX BDF3, respectively. The mentioned estimation of the numerical convergence rates is calculated for each method as
\begin{equation}\label{Conv. rate}
    \log_{\frac{h_{1}}{h_{2}} }  \left( \frac{ err_{h_{1}} }{ err_{h_{2}} } \right)
\end{equation}
where $err_{h_i}$ refers to the norm of the error for the step length $h_i$. We have used in this case $h_{1}=0.05$ because it is the largest step length for which both methods are stable, and $h_{2}=0.005$.

\end{sol}

\begin{obs}  \label{Remark-cond-stab}
Theorem \ref{Sufficient-cond-stab-simultaneous} provides a more restrictive condition than Proposition \ref{prop1} when talking about the step size. We know that the largest $h$ we can use in this example for stability is $h<h^{*}=0.157609$ for the IMEX BDF2 and $h<h^{*}=0.0760254$ for the IMEX BDF3. For both methods, $\mu_{i}\in D(0, \sigma_{z_{i}})\ \forall i=1,2,3,4$. Actually, $\mu_{2}$ lies on the boundary of the region mentioned for the first method cited, and also $\mu_{3}$, but referring now to the second method. Therefore, given the eigenvalues of A, $(\lambda_{1},\lambda_{2},\lambda_{3},\lambda_{4})=(3,8,17,30)$, the one that leads to the time restriction using Proposition \ref{prop1} is $\lambda_{2}$ for the IMEX BDF2 method and $\lambda_{3}$ for the IMEX BDF3 method, and not $\lambda_{4}$ as Theorem \ref{Sufficient-cond-stab-simultaneous} states. Consequently, the new restrictions given by the latest Theorem are $h<h^{*}_{new}=0.0420291$ and $h<h^{*}_{new}=0.0285391$, respectively, which implies choosing $h^{*}_{new}=\frac{4}{15}h^{*}$ for the IMEX BDF2 method and $h^{*}_{new}=\frac{17}{30}h^{*}$ for the IMEX BDF3 method, to ensure stability.

\vspace{0.1cm}
However, on practical problems coming from PDDEs the use of Theorem \ref{Sufficient-cond-stab-simultaneous} is much simpler than Proposition \ref{prop1}.  Additionally, we are not able to provide a result equivalent to  Proposition \ref{prop1} when $A$ and $B$ do not diagonalize simultaneously, but in the following Section we will obtain theorems equivalent to Theorem \ref{Sufficient-cond-stab-simultaneous}.
\end{obs}

 \section{Stability of IMEX-BDF methods in the general case}\label{stab-general-case}

 \subsection{Field of values and numerical radius of a matrix}\label{field-values}

Now, let us consider the case of a linear differential system of equations with delay, that is:
\begin{equation}\label{DDE eq LMM autonomous} \left\{ \begin{aligned} & y'(t)=-Ay(t) +By(t-\tau) & t\geq t_{0}\\ & y(t) = \phi(t) & t\leq t_{0}.
 \end{aligned} \right.
\end{equation} with $-A, B \in Mat_{d\times d}(\C)$.  We will impose that $-A$ be a negative definite Hermitian matrix, that is, its transposed conjugate satisfies $(-A)^{*}=-A$, and, in addition, $\langle x, (-A)x \rangle <0$ for all $x\neq 0, \ x \in \C^{d}$.

\vspace{0.1cm}
To analyze the stability of IMEX-BDF methods in general cases, we will exploit some connections between the regions analyzed in the scalar case and the Field Of Values (FOV) of a matrix $X$ \cite[Ch. I]{Horn-fieldofvalues}, which is defined as
	\begin{equation}\label{FOVdef}
	  F(X):=\left\{\langle \vec{x}, X \vec{x} \rangle : ||\vec{x}||^2 = 1, \ \vec{x} \in \mathbb{C}^d \right\}.
	\end{equation}
	In particular, we consider the set
	\begin{equation}\label{FOVexpr}
	       F_q(A,B):=\left\{\langle \vec{v}, A^{q-1}B\vec{v} \rangle : \langle \vec{v}, A^{q}\vec{v} \rangle = 1, \ \vec{v} \in \mathbb{C}^d, \ q\in \mathbb{R}\right\}.
	\end{equation}
	With $\vec{v}=A^{-\frac{q}{2}}\vec{x}$, it can be written as $ F_q(A,B)= F (A^{\frac{q}{2}-1} B A^{-\frac{q}{2}})$,
	where $F(X)$ is exactly the FOV defined in \eqref{FOVdef}. Note that we are considering $A$ (which is positive definite) to make the definition \eqref{FOVexpr} well-posed.
	
Using the change of variable $v=A^{-\frac{p}{2}}x$, we have that $$\langle v, A^{p-1}Bv \rangle=\langle A^{-\frac{p}{2}}x,A^{p-1}BA^{-\frac{p}{2}}x\rangle=(A^{-\frac{p}{2}}x)^{ *}A^{p-1}BA^{-\frac{p}{2}}x=x^{*}(A^{-\frac{p}{2}})^{*}A^{p-1}BA^{-\frac{p}{2}}x=$$ $$\underset{\underset{\text{A Hermitian}}{\uparrow}}{=}x^{*}A^{-\frac{p}{2}}A^{p-1}BA^{-\frac{p}{2}}x = x^{*}A^{\frac{p}{2}-1}B A^{-\frac{p}{2}} x = \langle x, A^{\frac{p}{2}-1}B A^{-\frac{p}{2}} x\rangle,$$
and hence, it is possible to rewrite the set $F_{p}$ as
$$F_{p}=F(A^{\frac{p}{2}-1}BA^{-\frac{p}{2}})=\left\{ \langle x,A^{\frac{p}{2}-1}BA^{-\frac{p}{2}}x\rangle \in \C \ / \ x\in \C^{d}, \ \|x\|=1 \right\}$$
being $F(A^{\frac{p}{2}-1}BA^{-\frac{p}{2}})$ the \textit{field of values of the matrix} $A^{\frac{p}{2}-1}BA^{-\frac{p}{2}}$.

\vspace{0.2cm}
The connections between the stability regions and the FOV have already been exploited in \cite{LMM-Seib-Ros} for the study of the unconditional stability of IMEX-BDF methods for PDEs without delay. In this section, we take inspiration from this study to analyze the conditional and unconditional stability of our IMEX-BDF methods when there is a delay.
	
Let us recall the analytic and geometric properties of the FOV that we will use.
\begin{remark}\label{remFOV}
Let $X$, $X_1$, $X_2$, be square matrices of size $d$, and $w_1$, $w_2 \ \in \mathbb{C}$:
\begin{enumerate}[label=(\alph*)]
    \item Compactness and convexity: $F(X)$ is a compact and convex set.   
    \item Spectral containment: \ $\sigma(X)\subset F(X)$.
    \item $ F ( X_1) + F( X_2) = \left\{x_1+x_2: x_1 \in F(X_1),\ x_2 \in F(X_2)\right\} $;
    \item $F( X_1) F( X_2) = \left\{x_1x_2: x_1 \in F(X_1),\ x_2 \in F(X_2)\right\} $;
    \item $F(w_1 X_1+w_2 I)=w_1 F(X_1)+w_2 \ \ $ \cite[p. 9, Properties 1.2.3 and 1.2.4]{Horn-fieldofvalues};
    \item $X$ is Hermitian if, and only if, $F(X)$ is a closed segment of the real line. 
    \item if $X_1$ is normal, then  $F(X_1) = \text{co}(\sigma(X_1))$ \cite[p. 11, Property 1.2.9]{Horn-fieldofvalues};  co$(\sigma(X))$ denotes the convex hull of the eigenvalues of the matrix $X$;
    \item as a direct consequence of \cite[Corollary 1.7.7]{Horn-fieldofvalues}, if $X_1$ is positive semi-definite, then $\sigma(X_1 X_2) \subseteq F(X_1) F(X_2)$.
\end{enumerate}

There exist several codes to efficiently compute the FOV; including the MATLAB \texttt{chebfun} routine \cite{chebfun}, based on a classical algorithm proposed by Johnson \cite{JohnsonSINUM1978}, but also there are similar codes in Mathematica, for example.
\end{remark}

Since the ``size'' of the field of values $F(\cdot)$ is of interest, because it can be used to know certain properties of the matrix considered, it is necessary to define the following notion: 
\begin{defn}\label{Def numerical radius}
    The numerical radius of a matrix $X \in Mat_{d\times d}(\C)$ is
    $$r(X)=\max\left\{ |c| \ / \ c\in F(X)\right\}.$$
\end{defn}
In this sense, the numerical radius of a given matrix will be, by definition, included in the circle centered at the origin of the complex plane, of radius the numerical radius of the same, that is, $F(\cdot)\subseteq D(0,r(\cdot))$, thus delimiting its size.

Likewise, it is also important to take into account the following definition:
\begin{defn}\label{Def spectral radius}
    We will call the spectral radius of a matrix $X \in Mat_{d\times d}(\C)$
    $$\rho(X)=\max\left\{|\lambda| \ / \ \lambda\in\sigma(X)\right\}.$$
\end{defn}

\vspace{0.1cm}

Taking all of that into account, the IMEX method (\ref{original IMEX-LMM}) applied now to the previous DDE \Eqref{DDE eq LMM autonomous} will be translated into:
\begin{equation}\label{autonomous IMEX-LMM}
\sum_{j=0}^{s} \alpha_{j} y_{n+j}=h\sum_{j=0}^{s}\beta_{j}(-A)y_{n+j}+h\sum_{j=0}^{s -1}\beta_{j}^{*} By_{n+j-m}
\end{equation}

However, taking the vector $Y_{n+s-1}=(y_{n+s-1},. ..,y_{n-m})$, it is possible to rewrite this multi-step method in the form $Y_{n+s}=W\cdot Y_{n+s-1}$, where the square matrix $W$ of dimension $s+m$ is defined as:
\begin{equation*} W={
   \begin{pmatrix} I\alpha_{s}-h(-A)\beta_{s }& 0 & 0 & \cdots & 0 & 0\\ 0 & I & 0 & \cdots & 0 & 0\\ 0 & 0 & \ddots & & & 0\\ \vdots & \vdots & & I & & \vdots \\ \vdots & \vdots & & & \ddots & 0\\ 0 & 0 & \cdots & \cdots & 0 & I
   \end{pmatrix}}^{-1}\cdot (W_{1}+ W_{2}+W_{3}),
   \end{equation*}
   being $I=I_{d}$, $0 = 0_d$ the null matrix with dimension $d$, and
   \begin{equation*} W_{1}=
   \begin{pmatrix} h(-A)\beta _{s-1}-I\alpha _{s-1} & h(- A)\beta _{s-2}-I\alpha _{s-2} & \cdots & h(-A)\beta _{0}-I\alpha _{0} & 0 & \overset{m}{\cdots } & 0\\ 0 & 0 & \cdots & 0 & 0 & \cdots & 0\\ \vdots & \vdots & & \vdots & \vdots & & \vdots\\ 0 & 0 & \cdots & 0 & 0 & \cdots & 0
 \end{pmatrix},
 \end{equation*}
 $$ W_{2}=\begin{pmatrix} 0 & \overset{m}{\cdots} & 0 & hB\beta^{*}_{s-1 } & hB\beta^{*}_{s-2} & \cdots & hB\beta^{*}_{0}\\ 0 & \cdots & 0 & 0 & 0 & \cdots & 0\\ \vdots & & \vdots & \vdots & \vdots & & \vdots\\ 0 & \cdots & 0 & 0 & 0 & \cdots & 0 \end{pmatrix},
  \quad \quad
  W_{3}=\begin{pmatrix} 0 & 0 & \cdots & \cdots & 0 & 0 \\ I & 0 & \cdots & \cdots & 0 & 0 \\ 0 & \ddots & & & \vdots & \vdots \\ \vdots & & I & & \vdots & \vdots \\ \vdots & & & \ddots & 0 & 0\\ 0 & \cdots & \cdots & 0 & I & 0 \end{pmatrix}.
$$

Let $V$ be an eigenvector of $W$ with eigenvalue $\xi$. Due to the shape of the matrix $W$, it is verified that, for some $v\neq 0, \ v\in \C^{d}$, the eigenvector can be expressed as:
$$V=\left(\xi^{s+m-1}v, \xi^{s+m-2},\cdots,\xi v, v\right)^{T} \ \in \C^{(s+m)d}$$

In turn, it is possible to write the characteristic equation of $W$ in the form:

$$\det(W-\xi I)=0 \ \Longleftrightarrow \ \det\left(\xi^{m}\left(\frac{1}{h}\rho(\xi)I-\sigma^{*}(\xi)(-A)\right)-\sigma(\xi)B\right)=0$$
where $\rho(\xi), \sigma(\xi), \sigma^{*}(\xi)$ are defined in (\ref{polinomios IMEX}).

\vspace{0.1cm}
With all this we can affirm that, if $\xi$ is an eigenvalue of $W$, then there will always exist an eigenvector $V$ (whose structure we have previously defined), such that:
$$T(\xi)v=0, \quad \text{with} \quad T(x)=x^{m}\left(\frac{1}{h}\rho(x)I-\sigma(x)(-A)\right)-\sigma^{*}(x)B.$$

We have thus obtained a characterization of the stability of the method (\ref{autonomous IMEX-LMM}) from the characteristic equation of the matrix $W$, which is contained in the following proposition:
\begin{prop}
(see \emph{\cite[Proposition 1]{Art-Alejandro},\cite[Proposition 2]{LMM-Seib-Ros}).} If the matrix $T(x)$ is non-singular for all $|x|\geq 1$, then the method (\ref{autonomous IMEX-LMM}) is stable.
\end{prop}

However, the matrix $T(x)$ relates the time coefficients $\alpha_{j}, \beta_{j}, \beta^{*}_{j}$ with the matrices $- h A, h B$. To simplify the problem, we will try to reduce it to study stability conditions in scalar equations.

Let $p\in\R$. Multiplying by the matrix $A^{p-1}$, which is positive definite since $A$ is positive, on the left side of $T(\xi)v=0$, we arrive at the equality:
$$\frac{1}{h}\xi^{m}\rho(\xi)A^{p-1}v+\xi^{m}\sigma(\xi) A^{p}v-\sigma^{*}(\xi)B A^{p-1}v=0.$$

Multiplying now by $v^{*}$ on the right, and recalling that $\langle v_{1}, v_{2}\rangle=v_{1}^{*} v_{2}$ for all $v_{1},v_{2}\in\C^{d}$ and defining in this case
$$z=-h\cdot\frac{\langle v, A^{p}v\rangle}{\langle v, A^{p-1}v\rangle}, \quad \quad \mu=\frac{\langle v, A^{p-1}Bv\rangle}{\langle v, A^{p}v\rangle},$$ we obtain the equation \begin{equation}\label{IMEX-LMM eq autonomous character} \zeta^{m}(\rho(\zeta)-z\sigma(\zeta))+z\mu\sigma^{*}(\zeta)=0,
\end{equation}
for these new values of $z \in F( A h)$ and $\mu \in F_p$.

\vspace{0.1cm}
With all that has been stated so far, we can deduce the stability of a linear system, which is summarized as follows:
\begin{thm}\label{Th: stability of the autonomous system and h}
    Let us consider the system (\ref{eq:1:EDR_EstabilidadMétodo}), and let $\lambda_{d}$ be the largest eigenvalue of the matrix $A$. Suppose there exists $p\in\R$ such that $F_{p}\subseteq D(0,r)$, where $0<r\leq 1$, then the IMEX BDF2 and IMEX BDF3 methods applied to the linear system (\ref{eq:1:EDR_EstabilidadMétodo}) are stable for $h\leq h^{*}$ with $h^{*}=   \frac{  | \chi(r ) | }{\lambda_{d}}$ for IMEX BDF 2 and  $h^{*}=   \frac{  | \tilde{ \chi}(r) | }{\lambda_{d}}$ for IMEX BDF 3.
\end{thm}
\begin{dem}  Let the equation be (\ref{IMEX-LMM eq autonomous character}) where we remember that the parameters $z$ and $\mu$ follow the formulas:
\begin{equation}\label{condition-stability}
    z=-h\cdot\frac{\langle v, A^{p}v\rangle}{\langle v, A^{p-1}v\rangle}, \quad \quad \mu=\frac{\langle v, Bv\rangle}{\langle v, A^{p}v\rangle}, \quad \text{with} \ v\in\C^{d}.
\end{equation}
From this, and using certain properties of the matrix $A$, we were able to previously prove that $\mu\in F_{p}=F(A^{\frac{p}{2}-1}BA^{-\frac{p}{2}})$. Let us also consider Proposition \ref{Prop D(0,sigma)cD_z}, which stated that $D(0,\sigma_{z})\subseteq D_{z}$ for $z\in \R_{<0}\cup\{-\infty\}$. Thus, taking into account both implications, we can declare that whenever $F_{p}\subseteq D(0,\sigma_{z})$, then $\mu\in D_{z}$.

However, this must be true for any value $v\in \C^{d}$. Therefore, it is sufficient that
$$F_{p} \subseteq \underset{z\in F(hA)}{\bigcap} D_{z} $$
is satisfied to obtain stability.

By the property (f) of Remark \ref{remFOV} we have that the numerical range of a Hermitian matrix is a segment of the real line. Specifically $F(hA)=[-\lambda_{d}h, -\lambda_{1}h]$, where $\lambda_{1}<...<\lambda_{d}$ are the eigenvalues ordered in increasing order of the matrix $A$. With this, $z_{d}=-\lambda _{d} h$ will be the minimum value among all $z$, and then
$$D(0,\sigma _{z_{d}})\underset{\underset{\text{Lem } \ref{Observation}}{\uparrow}}{\subseteq} D(0,\sigma _{z}) \underset{\underset{\text{Cor } \ref{Prop D(0,sigma)cD_z} }{\uparrow}}{\subseteq} D_{z} \quad \forall \ z\in F(hA) \ \Longrightarrow \ D(0,\sigma _{z_{d}})\subseteq \underset{z\in F(hA)}{\bigcap}D_{z} \quad \forall \ z\in F(hA).$$ In this sense, we can reduce the problem to the study when $F_{p}\subseteq D(0,\sigma_{z_{d}})$, where $z_{d}=-\lambda_{d}h$ with $\lambda_{d}=\rho(A)$, that is, the spectral radius of the matrix $A$.
$\hfill\square$
\end{dem}

\begin{coro}\label{Coro: reg est incondicional}
The unconditional stability regions of the IMEX BDF2 and IMEX BDF3 methods contain, respectively, the disks $D(0,\frac{1}{3})$ and $D(0,\frac{1}{7})$.
\end{coro}
\begin{dem}

Taking into account the Corollary \ref{Prop D(0,sigma)cD_z} and Lemma \ref{Prop D(0,sigma)cD_z}, we quickly arrive at $D(0,\sigma_{-\infty})\subseteq D_{-\infty}=D$. Starting from the minimum functions of both methods, considering specifically $\psi(z)$ and $\tilde{\psi}(z)$, and calculating $\lim_{z\rightarrow -\infty} \psi (z)$, $\lim_{z\rightarrow -\infty} \tilde{\psi} (z)$, we obtain that the minimum values  $\sigma_{-\infty}$ are $\frac{1}{3}$ for the IMEX BDF2 and $\frac{1}{7}$ for the IMEX BDF3, which concludes. $\hfill\square$
\end{dem}

\vspace{\baselineskip}
Until now, the only conditions imposed on the matrices $-A$ and $B$ were that the first one was hermitian and negative definite. Let us now suppose that additionally the matrix $A^{-1}B$ is normal, which implied fulfilling $(A^{-1}B)^{*}A^{-1}B=A^{-1}B(A^{-1}B)^{*}$,
where $(A^{-1}B)^{*}$ is the conjugate transpose of $A^{-1}B$. Hence we have in \cite[p.45]{Horn-fieldofvalues} that the numerical radius (Definition \ref{Def numerical radius}) and the spectral radius (Definition \ref{Def spectral radius}) of the mentioned matrix coincide: \quad $r(A^{-1}B)=\rho(A^{-1}B).$

Furthermore
$$\rho(A^{-1}B)\leq \underset{\Re(\xi)=0}{\sup}\rho((\xi I-A)^{-1}B).$$

Then, if we impose the stability condition
\begin{equation}\label{Condición estabilidad asintótica S.A.L}
    \left\{
     \begin{aligned}
         & \lambda\in \sigma[L] \ \Rightarrow \ \Re(\lambda)<0\\
         & \rho[(\xi I-L)^{-1}G]<1 \ \text{whenever} \ \Re(\xi)=0,\ \xi\neq 0\\
         & -1 \notin \sigma[L^{-1}G]
     \end{aligned}
     \right.
\end{equation}
derived in \cite[Theorem 3.1]{Stability-theta-methods-DDE} and relative to the stability of the solutions in linear systems, it leads us to $\ F(A^{-1}B)\subseteq D(0,1) \ $. This is because, by definition, $r(A^{-1}B)$ is, in absolute value, the largest of all the values of $F(A^{-1}B)$, and therefore, $$F(A^{-1}B)\ \subseteq \ D(0,r(A^{-1}B))=D(0,\rho(A^{-1}B)) \ \subseteq \ D(0,1),$$

With all this, whenever problem \Eqref{eq:1:EDR_EstabilidadMétodo} is stable, taking into account Theorem \ref{Th: stability of the autonomous system and h}, we affirm that the methods will be stable if $A^{-1}B$ is normal, and $h\leq h^{*}$, where $h^{*}=\frac{1}{\lambda_{d}\sqrt{2}}$ in the IMEX BDF2, and $h^{*}=\frac{0.722965}{\lambda _{d}}$ in the IMEX BDF3, with $\lambda _{d}=\rho(A)$.  The reason is $|  \chi(1) | = \frac{1}{ \sqrt{2}}$ and $| \tilde{ \chi}(1) | = 0.722965$.

\subsection{ Example: linear system where matrices do not simultaneously diagonalize }

As a practical example of how these theoretical concepts can be employed, let us consider the following test example:
\begin{ej}
\textit{Stability restrictions of  IMEX BDF2 and IMEX BDF3 methods applied to the linear system:}
\begin{equation*}
\left\{
\begin{aligned}
& y'(t)=-A\ y(t)+B\ y(t-1)+\tilde{f}(t) & t\geq 0\\
& y(t) = \begin{pmatrix}
\cos(t) \\
e^{-0.1t} \\
1+t
\end{pmatrix}
& t\leq 0
\end{aligned}
\right.
\end{equation*}
\textit{where the matrices are defined as}
$$
A=\begin{pmatrix}
20 & -4 & 0 \\
-4 & 20 & 0\\
0 & 0 & 10
\end{pmatrix},
\quad \quad
B=\begin{pmatrix}
-2 & 1 & 0 \\
-1 & -2 & 0\\
0 & 1 & 6
\end{pmatrix}.
$$
\textit{and $\Tilde{f}$ is such that $y(t) =(\cos(t),e^{-0.1t}, 1+t)^{T}$ is the solution of the differential equation with delay.}
\end{ej}
\begin{sol}

Reasoning in a similar way to the previous example, let us first calculate
$\Tilde{f}(t)$ starting from the fact that $\Tilde{f}(t)=y'(t)+Ay(t)-By(t-1)$ and knowing the real solution of the example. Thus, we obtain:
\begin{equation*}
\begin{aligned}
\Tilde{f}(t)=& ( -e^{-0.1 (t-1)}-4 e^{-0.1 t}-\sin (t)+2 \cos (1-t)+20 \cos (t), \\ & 2 e^{-0.1 (t-1)}+19.9 e^{-0.1 t}+\cos (1-t)-4 \cos (t), \\ & 1 -6 t-e^{-0.1 (t-1)}+10 (t+1) ).
\end{aligned}
\end{equation*}

Regarding the IMEX BDF2 and IMEX BDF3 methods, we will again use the formulas depicted on the previous example (\ref{IMEX BDF2 ejemplos})-(\ref{IMEX BDF3 ejemplos}), obtained by applying the general equations (\ref{IMEX-BDF2 general})-(\ref{IMEX-BDF3 general}) to this linear system.

\vspace{\baselineskip}
Now, let us verify if the system presented in the example satisfies the conditions required to apply our theorems. First, it is easy to confirm that the matrix $-A$ is hermitian and negative definite, using some simple calculations and the fact that the eigenvalues of $A$ are $(24,16,10)$. On the other hand, we quickly observe that $-A$ and $B$ do not diagonalize simultaneously because the matrices do not commute. Consequently we can see that the analysis of this equation can be reasoned using the theoretical concepts developed in the previous chapter. In particular, we will use Theorem \ref{Th: stability of the autonomous system and h}, which shows the procedure to follow in order to find the step size restrictions $h^{*}$ for both methods in the case in which we find ourselves ($A$ and $B$ do not diagonalize simultaneously).

Let us consider for simplicity that $p=0$. Applying the mentioned theorem will consist of finding the value $r$ such that $F((-A)^{-1}B)\subseteq D(0,r)$. In Fig. \ref{fig:rangonum} a graphical representation of $F((-A)^{-1}B)$ is shown.

\begin{figure}[h!]
    \centering
    \begin{subfigure}[b]{0.34\textwidth}
    \centering
    \includegraphics[width=\textwidth]{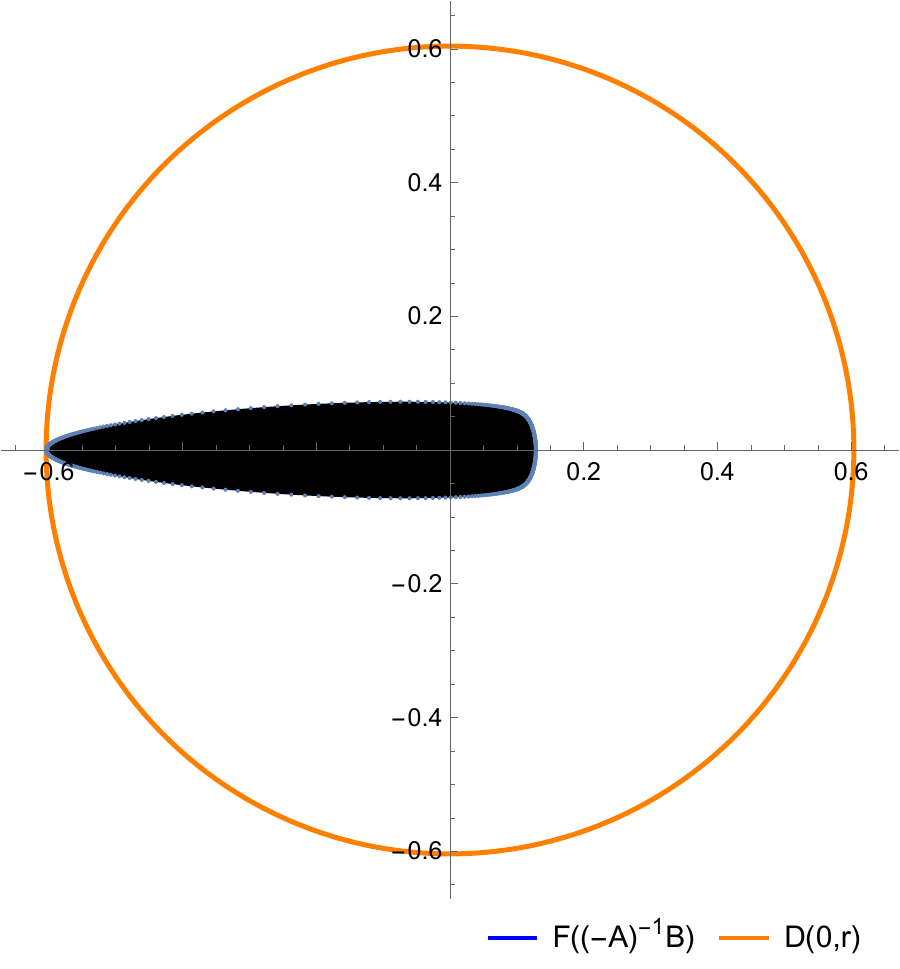}
    \label{fig:figura15}
    \end{subfigure}
    \caption{Location of the numerical range $F((-A)^{-1}B)$.}
    \label{fig:rangonum}
\end{figure}

\vspace{0.15cm}
From this we affirm that the radius $r$ of the disk centered on zero that contains the numerical range, corresponds to the largest vector, in absolute value, of $F((-A)^{-1}B)$. This parameter for this particular problem is $r=0.604$, and reasoning from it and following the mentioned theorem, we arrive at the step size restrictions for this linear differential equation with delay: $h\leq h^{*}=0.164762$ for the IMEX BDF2 and $h\leq h^{*}=0.059618$ for the IMEX BDF3.

\vspace{0.15cm}
In Table \ref{tab:Autonomous error table 2},  the errors committed in each method, at time $t_{e}=500$ for different step sizes, are provided, as well as their convergence rates, which are calculated again using (\ref{Conv. rate}) with $h_{1}=0.05$ and $h_{2}=0.005$.\\
We can observe numerically (see Table \ref{tab:Autonomous error table 2}), that IMEX BDF2 is stable for $h=0.25 >h^*$, but unstabilities appear in the solution with $h=0.5$.  Additionally, IMEX BDF3 is stable for $h=0.1 >h^*$, but it is unstable with $h=0.25$.    The explanation is similar to the one given in Remark  \ref{Remark-cond-stab}:   Theorem  \ref{Th: stability of the autonomous system and h} only provides sufficient conditions to obtain stability, because it is sufficient that $\mu\in D_{z}$ for any $\mu$ and $z$ given by equation \Eqref{condition-stability}, but since all these calculations would be very expensive, we replace this condition by imposing
$$F_{p} \subseteq \underset{z\in F(hA)}{\bigcap} D_{z}. $$
\begin{table}[H]
\centering
\resizebox{\textwidth}{!}{%
\begin{tabular}{| c | c | c |} \hline & IMEX BDF2 & IMEX BDF3\\ \hline
$ h=0.5 $ & $ (1.4315\cdot 10^{-2}, 9.9801\cdot 10^{-3}, 7.7264\cdot 10^{2}) $ & $(1.0906\cdot 10^{-2}, 6.8539\cdot 10^{-3}, 6.9434\cdot 10^{80}  )$
\\ $h=0.25$ & $(4.6735\cdot 10^{-3}, 3.1262\cdot 10^{-3}, 3.3401\cdot 10^{-4})$ & $(1.1014\cdot 10^{-3}, 6.8547\cdot 10^{-4}, 8.8856 \cdot 10^{38} )$
 \\ $h=0.1$ & $(8.3395\cdot 10^{-4}, 5.5237\cdot 10^{-4}, 4.2632\cdot 10^{-5})$ & $(5.3865\cdot 10^{-5}, 3.2963\cdot 10^{-5}, 5.3185 \cdot 10^{-6})$
 \\ $h=0.05$ & $ ( 2.1457\cdot 10^{-4},1.4179\cdot 10^{-4}, 9.8392\cdot 10^{-6})$ & $ (5.9368 \cdot 10^{-6}, 3.6030\cdot 10^{-6}, 7.0573 \cdot 10^{-7} )$
 \\ $h=0.025$ & $ (5.4342\cdot 10^{-5}, 3.5873\cdot 10^{-5}, 2.3623\cdot 10^{-6} )$ & $ (6.9036 \cdot 10^{-7}, 4.1690\cdot 10^{-7}, 9.0602 \cdot 10^{-8})$
 \\ $h=0.01$ & $(8.7586 \cdot 10^{-6}, 5.7784 \cdot 10^{-6}, 3.6884 \cdot 10^{-7} )$ & $ (4.2159 \cdot 10^{-8}, 2.5375 \cdot 10^{-8}, 5.8877\cdot 10^{-9} )$
 \\ $h=0.005$ & $ (2.1947 \cdot 10^{-6}, 1.4477 \cdot 10^{-6}, 9.1462 \cdot 10^{-8} )$ & $ ( 5.1848 \cdot 10^{-9}, 3.1170 \cdot 10^{-9}, 7.4311 \cdot 10^{-10}  )$
\\ \hline
Conv. Rate &  1.99045  & 3.0589
\\ \hline
\end{tabular}
}
\caption{Errors committed at time $t_{e}=500$ for $y(t)=(\cos(t),e^{-0.1t}, 1+t)^{T}$.}
\label{tab:Autonomous error table 2}
\end{table}

\end{sol}

\section{Numerical Simulations}\label{Numer-simulations}

\subsection{ Linear system of PDDEs }

Let us develop the following numerical simulation related to a partial differential equation with delay, following, to a certain extent, the methodology used in the previous sections for systems of differential equations.

Let us consider the following system of PDDEs:
 \begin{equation}\label{PDDE u1u2}
        \left\{
            \begin{aligned}
              & \frac{\partial u_{1}}{\partial t}=a_{1}\frac{\partial^{2}u_{1}}{\partial x^{2}}-e^{l\frac{\pi}{2}}u_{1}(t-\tau,x)+\frac{1}{4}e^{l\frac{\pi}{2}}(4l+\pi^{2})u_{2}(t-\tau,x) & t\geq 0, & \quad -1\leq x\leq 1\\
              & \frac{\partial u_{2}}{\partial t}=a_{2}\frac{\partial^{2}u_{2}}{\partial x^{2}}-e^{l\frac{\pi}{2}}u_{2}(t-\tau,x)-\frac{1}{4}e^{l\frac{\pi}{4}}(4l+\pi^{2})u_{1}(t-\tau,x)  & t\geq 0, & \quad -1\leq x\leq 1
            \end{aligned}
        \right.
    \end{equation}
    with the initial and boundary conditions
     \begin{equation*}
        \left\{
            \begin{aligned}
              & u_{1}(t,-1)=u_{1}(t,1)=0\\
              & u_{2}(t,-1)=u_{2}(t,1)=0\\
              & u_{1}(t,x)=\phi_{1}(t,x) \quad t\leq 0\\
              & u_{2}(t,x)=\phi_{2}(t,x) \quad t\leq 0,
            \end{aligned}
        \right.
    \end{equation*}
and being $a_{1}, a_{2}$ positive numbers, $l$ a real number, and $\phi_{1}(t), \phi_{2}(t)$ two given functions.

\begin{sol}

Since our theory is based on differential equations that depend only on the parameter $t$ and, in this case, the partial derivatives are also calculated with respect to the variable $x$, let us apply the method of lines (MOL) to the PDDEs of the statement. This procedure consists of replacing the derivatives of the spatial variables with numerical approximations of these at certain nodes, so that only the derivatives with respect to $t$ remain, and thus a system of DDEs is obtained.

Let $n\in\N$ and $x_{0}=-1$ and $x_{n}=1$. From this we can define $\Delta x=\frac{x_{n}-x_{0}}{n}=\frac{2}{n}$ such that we have the mesh points $x_{0}=-1<\cdots<x_{j}=x_{0}+j\cdot \Delta x <\cdots<x_{n}=1$. In this case we want to approximate the second order derivative so we consider, for example, the second order centered finite difference formula:
\begin{equation}\label{second_order_finite_diference_formula}
    f''(x_{s})=\frac{f(x_{s}-\Delta x)-2f(x_{s})+f(x_{s}+\Delta x)}{(\Delta x)^{2}}-\frac{(\Delta x)^{2}}{12}f^{(4)}(\xi), \quad \xi\in(x_{s}-\Delta x,x_{s}+\Delta x).
\end{equation}
In this way, for every point of the mesh $x_{j}\in\{x_{j}\}$ the following holds:
\begin{equation}\label{second_order_finite_diference}
    \left.\frac{\partial^{2}u(t,x)}{\partial x^{2}}\right|_{x=x_{j}}=\frac{u_{j-1}(t)-2u_{j}(t)+u_{j+1}(t)}{(\Delta x)^{2}}+O((\Delta x)^{2}),
\end{equation}
and therefore, we can approximate the previous system of PDDEs by the system of DDEs
   \begin{equation*}
        \left\{
            \begin{aligned}
              & \frac{\partial u_{1,j}(t)}{\partial t}=a_{1}\frac{u_{1,j-1}(t)-2u_{1,j}+u_{1,j+1}(t)}{(\Delta x)^{2}}-e^{l\frac{\pi}{2}}u_{1,j}(t-\tau,x)+\frac{1}{4}e^{l\frac{\pi}{2}}(4l+\pi^{2})u_{2,j}(t-\tau,x) \\
              & \frac{\partial u_{2,j}(t)}{\partial t}=a_{2}\frac{u_{2,j-1}(t)-2u_{2,j}+u_{2,j+1}(t)}{(\Delta x)^{2}}-e^{l\frac{\pi}{2}}u_{2,j}(t-\tau,x)-\frac{1}{4}e^{l\frac{\pi}{2}}(4l+\pi^{2})u_{1,j}(t-\tau,x)
            \end{aligned}
        \right.
    \end{equation*}
    for  $t\geq 0$ and $-1\leq x\leq 1$, and with the initial and boundary conditions
      \begin{equation*}
        \left\{
            \begin{aligned}
              & u_{1,-1}(t)=u_{1,1}(t)=0,\\
              & u_{2,-1}(t)=u_{2,1}(t)=0,\\
              & u_{1,j}(t)=\phi_{1}(t,x_{j}) \quad t\leq 0,\\
              & u_{2,j}(t)=\phi_{2}(t,x_{j}) \quad t\leq 0.
            \end{aligned}
        \right.
    \end{equation*}

Taking $z_{k,j}^{i}$, for each $k=1,2$, as the approximation of $u_{k}(t_{i},x_{j})$ and $z_{k}(t)=(z_{k,1}(t),...,z_{k,n-1}(t))^{T}$, we have that the above system simplifies to the equation
\begin{equation}\label{DDEeq3}
    z'(t)=Az(t)+Bz(t-\tau)
\end{equation}
where we define $z(t)=(z_{1}(t),z_{2}(t))$ and the matrices
\begin{equation*}
A=\begin{pmatrix}
a_{1}T & 0 \\
0 & a_{2}T
\end{pmatrix},
\quad \quad
B=e^{l\frac{\pi}{2}}\begin{pmatrix} -I & (l+\frac{\pi^{2}}{4})I\\ -(l+\frac{\pi^{2}}{4})I & -I \end{pmatrix}, \quad \quad I=I_{n-1}, \end{equation*} \begin{equation*} T=\frac{1}{(\Delta x)^{2}} \begin{pmatrix} -2 & 1 & 0 & 0 & \cdots & 0 \\ 1 & -2 & 1 & 0 & \cdots & 0 \\ 0 & 1 & -2 & \ddots & & \vdots \\ \vdots & & \ddots & \ddots & \ddots & 0\\ \vdots & & & \ddots & -2 & 1\\
0 & \cdots & \cdots & 0 & 1 & -2
\end{pmatrix}.
\end{equation*}
$T$ is known as a Toeplitz Matrix, and is characterized by being a square matrix in which the elements of its diagonals are constant. Generally, the eigenvalues $\lambda_{r}$ of a Toeplitz matrix of the form we find in this case, that is
\begin{equation*}
    M=\begin{pmatrix}
        \alpha & \beta & 0 & 0 & \cdots & 0 \\
        \gamma & \alpha & \beta & 0 & \cdots & 0 \\
        0 & \gamma & \alpha & \ddots &  & \vdots \\
        \vdots &  & \ddots & \ddots & \ddots & 0\\
        \vdots &  &  & \ddots & \alpha & \beta\\
        0 & \cdots & \cdots & 0 & \gamma & \alpha
    \end{pmatrix}_{N\cdot N}
    ,
\end{equation*}
are
\begin{equation}\label{eigenvaluesT}
    \lambda_{r}=\alpha-2\sqrt{\beta \gamma}\cos{\left(\frac{r\pi}{N+1}\right)}, \quad r=1,...,N.
\end{equation}
Specifically for this matrix T, we have that $\alpha=-\frac{2}{(\Delta x)^{2}}$, $\beta=\gamma=\frac{1}{(\Delta x)^{2}}$ and $N=n-1$, so $\lambda_{r}< 0 \ \forall r$. In this sense, $A$ will be a definite negative hermitian matrix whenever $a_1, a_2 >0$.

\vspace{\baselineskip}
Thanks to all of this, it is possible to calculate a numerical solution of the PDDEs of the statement by applying different methods to the previous system (\ref{DDEeq3}). Let us use the IMEX BDF2 and IMEX BDF3 methods developed in more detail above, which will follow again the formulas (\ref{IMEX BDF2 ejemplos})-(\ref{IMEX BDF3 ejemplos}), but considering here $\Tilde{f}(t)=0$. Let, for example, $a_{1}=1$, $a_{2}=1$, $n=100$ and $\tau=1$. Initial and boundary conditions are chosen such that the solution of these PDDEs are known as:
\begin{equation*}
\left\{
\begin{aligned}
& u_{1}(t,x)=\phi_{1}(t,x)=e^{lt}\sin(t)\cos\left(\frac{\pi}{2}x\right),\\
& u_{2}(t,x)=\phi_{2}(t,x)=e^{lt}\cos(t)\cos\left(\frac{\pi}{2}x\right).
\end{aligned}
\right.
\end{equation*}
Therefore we can affirm that the sign of the variable $l$ determines the result of the solution, since it diverges for positive values of this parameter, and converges to zero for negative values.

\vspace{\baselineskip}
Since this is again a linear system, and $A$ and $B$ do not commute, that is, the matrices do not diagonalize simultaneously ($A$ has all its eigenvalues simple), using Theorem \ref{Th: stability of the autonomous system and h}, we can study the stability of IMEX BDF2 and IMEX BDF3 methods. That is, let us consider, for example, $p=0$ and let us see where $F(A^{-1}B)$ is contained for the matrices $A$ and $B$ considered above.  In Fig. \ref{FieldA1B} these regions are provided for $l=-0.75$ and $l=0.75$.

\begin{figure}[h!]
    \centering
    \begin{subfigure}[b]{0.34\textwidth}
             \centering
             \includegraphics[width=\textwidth]{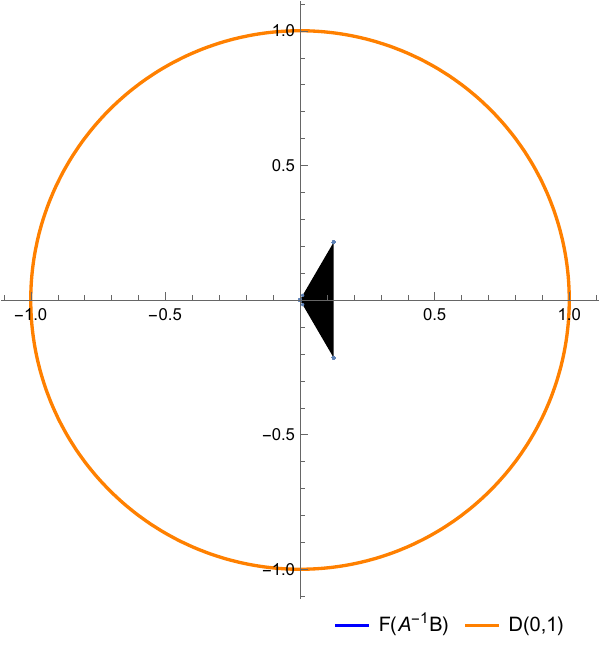}
             \caption{$l=-0.75$}
             \label{fig:figura16}
    \end{subfigure}
    \quad \quad
    \begin{subfigure}[b]{0.34\textwidth}
             \centering
             \includegraphics[width=\textwidth]{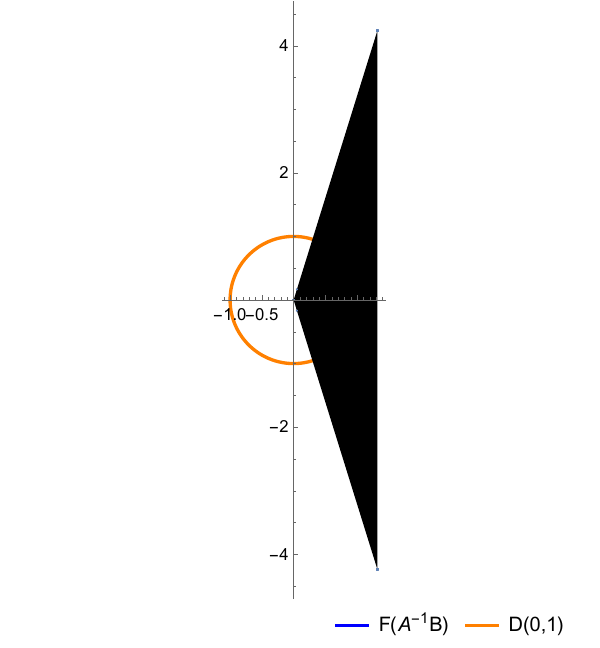}
             \caption{$l=0.75$}
             \label{fig:figura17}
    \end{subfigure}
    \caption{Location of the numerical range $F(A^{-1}B)$ for $n=100$ and different values of $l$.}\label{FieldA1B}
\end{figure}

From this, it is possible to obtain the value of the step size constraint $h^{*}$ for this problem. That is, taking into account Theorem \ref{Th: stability of the autonomous system and h} and the previous representation, we know that by imposing $l=-0.75$ and $n=100$, the numerical range $F(A^{-1}B)$ is contained in the disk $D(0,0.2142)$. For the IMEX BDF2 method we are in the unconditional stability region, so this procedure will be stable for any imposition on the step size $h$. As for IMEX BDF3, we are in the interval $\frac{1}{7}<r<0.218109$, so the step restriction will be given by the inequality $\quad h\leq \frac{20}{3\lambda_{d}(-1+7r)} \quad$
which leads to, for this method and the established parameters, $h^{*}=0.00133526$.

\vspace{0.1cm}
In Figures \ref{fig:solu1u2BDF2} and \ref{fig:solu1u2BDF3} the solutions obtained by each method for each variable at time $t_{e}=60$ are provided.
\begin{figure}[h!]
    \centering
    \begin{subfigure}[b]{0.32\textwidth}
             \centering
             \includegraphics[width=\textwidth]{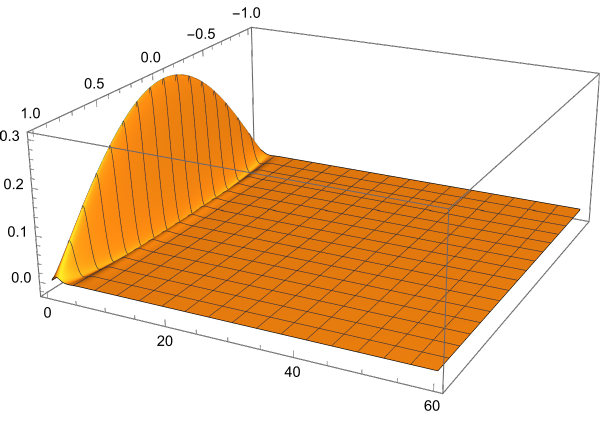}
             \caption{$u_{1}$}
             \label{fig:solu1BDF2}
    \end{subfigure}
    \begin{subfigure}[b]{0.32\textwidth}
             \centering
             \includegraphics[width=\textwidth]{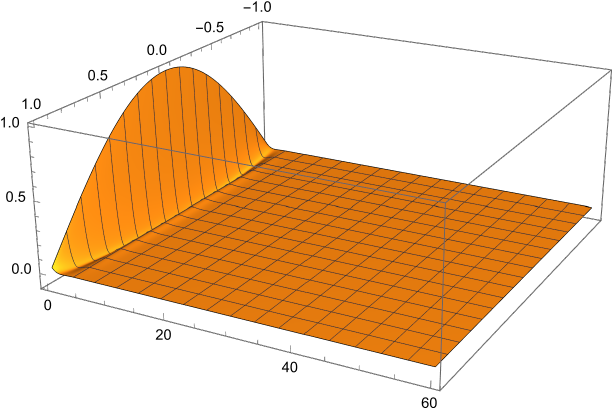}
             \caption{$u_{2}$}
            \label{fig:solu2BDF2}
    \end{subfigure}
    \begin{subfigure}[b]{0.33\textwidth}
             \centering
             \includegraphics[width=\textwidth]{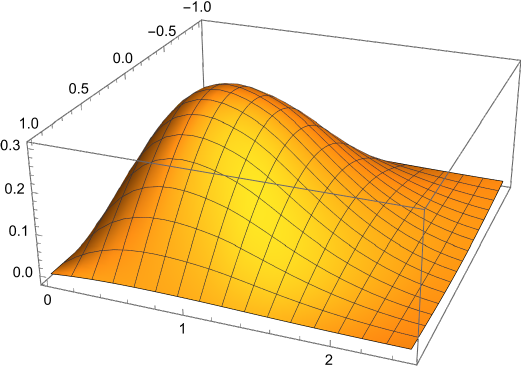}
             \caption{Detail of $u_{1}$, $t \in [0, 2.5]$.}
            \label{fig:solu1BDF2b}
    \end{subfigure}
    \caption{Numerical solution of (\ref{PDDE u1u2}) given by the IMEX BDF2 method with $h=0.1$ and $t_{e}=60$((a) and (b)), and detail of the numerical solution with $t \in [0, 2.5]$.}\label{fig:solu1u2BDF2}
\end{figure}
\begin{figure}[h!]
    \centering
    \begin{subfigure}[b]{0.32\textwidth}
             \centering
             \includegraphics[width=\textwidth]{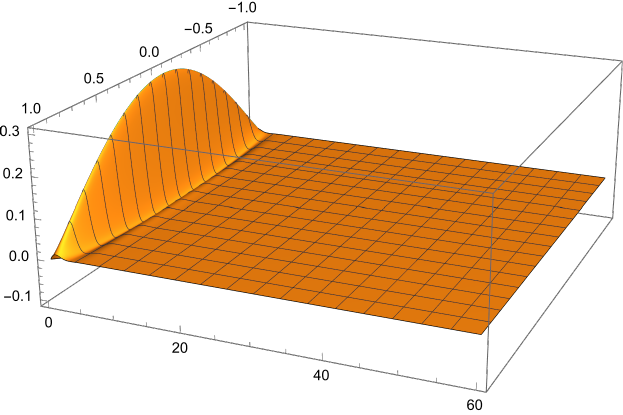}
             \caption{$u_{1}$}
             \label{fig:solu1BDF3}
    \end{subfigure}
    \begin{subfigure}[b]{0.32\textwidth}
             \centering
             \includegraphics[width=\textwidth]{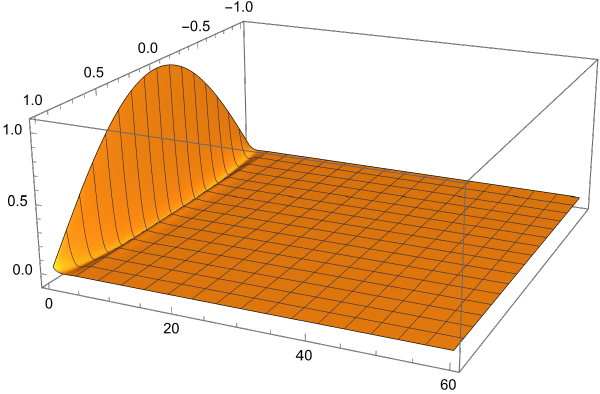}
             \caption{$u_{2}$ }
            \label{fig:solu2BDF3b}
    \end{subfigure}
    \begin{subfigure}[b]{0.33\textwidth}
             \centering
             \includegraphics[width=\textwidth]{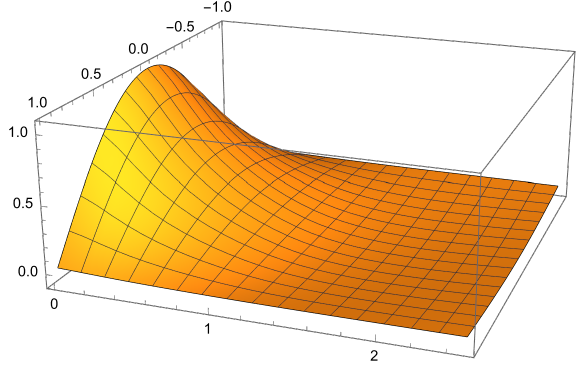}
             \caption{ Detail of $u_{2}$, $t \in [0, 2.5]$. }
            \label{fig:solu2BDF3}
    \end{subfigure}
    \caption{Numerical solution of (\ref{PDDE u1u2}) given by the IMEX BDF3 method with $h=0.1$ and $t_{e}=60$((a) and (b)), and detail of the numerical solution with $t \in [0, 2.5]$.}\label{fig:solu1u2BDF3}
\end{figure}

It is remarkable that, even if we impose step sizes larger than the existing restriction for $h$ in the IMEX BDF3,  we find very small errors, so that even for values larger than $h^{*}$ stability could be achieved. In this sense, the Theorem \ref{Th: stability of the autonomous system and h} guarantees us a sufficient, but not necessary, condition for stability.\\
In turn, it should be noted that, if we had established $l=0.75$, either of the two methods would return, for all $h$, approximations such that they tend to infinity, in accordance with the solution of the problem.

\vspace{0.2cm}
We have thus obtained an approximation of the solution of the problem, and specifically an analysis of the error committed when applying the IMEX BDF2 and IMEX BDF3 methods to a system of partial differential equations with delay.

\end{sol}

\subsection{ A non-linear PDDE: a delayed variant of the forced Burger's equations }

We consider now a variant of the Burgers' equation \cite[Subsec. 3.2]{Calvo2021} with delay in the convective term, plus a forcing term
	\begin{equation}\label{varBurg}
		\frac{\partial u}{\partial t}=\epsilon \frac{\partial^2 u}{\partial x^2}-\frac{1}{2} \left( u(t-\tau,x)^2  \right)_x +\tilde{f}(t,x)  , \quad \text{where } (t,x) \in [0,  t_e] \times [0,1],
	\end{equation}
    equipped with homogeneous Dirichlet conditions ($u(t,0) = u(t,1)=0$) and initial condition (IC)
    \begin{equation*}
               u(t,x) = \sin (\pi x), \quad t \leq 0.
    \end{equation*}
    We have chosen $\tilde{f}(t,x) =10 x (1-x) ( 1 + x \sin (t x))$ to study the stability of the problem  with a solution which does not decay and has some variability in a large interval.  Thus, we can choose $\epsilon =1$,  $n=100$, $\tau=1$ and $t_e=20$.

\begin{sol}

Let us proceed in the same way as in the previous example, but this time taking into account the forcing term of the differential equation.

Let $n\in\N$, $x_{0}=0$ and $x_{n}=1$, from which we can define $\Delta x=\frac{x_{n}-x_{0}}{n}$ and the mesh points $x_{0}=0<\cdots<x_{j}=j\cdot \Delta x<\cdots<x_{n}=1$. In order to approximate the different derivatives let us again consider for the second order derivative the second order centered finite difference formula (\ref{second_order_finite_diference_formula}), thanks to which, for every point of the mesh, (\ref{second_order_finite_diference}) holds. Knowing also that
$$\frac{\partial u(t-\tau,x)^{2}}{\partial x}=2\cdot u(t-\tau,x)\cdot u(t-\tau,x)_{x},$$
where we approximate the last factor with the first order central difference formula
$$f'(x_{s})\approx\frac{f(x_{s}+\Delta x)-f(x_{s}-\Delta x)}{2\Delta x}-\frac{(\Delta x)^2}{6}f'''(\xi)\quad \xi\in(x_{s}-\Delta x,x_{s}+\Delta x),$$
we have, for every $x_{j}\in\{x_{j}\}$, that
$$\left.\frac{\partial u(t-\tau,x)^{2}}{\partial x}\right|_{x=x_{j}}=  u(t-\tau,x_{j})\cdot \frac{u(t-\tau,x_{j+1})-u(t-\tau,x_{j-1})}{  \Delta x}+O((\Delta x)^2).$$

\vspace{\baselineskip}
Thanks to all of this, the PDDE presented on the statement above can be approximated by the DDE
\begin{equation} \label{systemPr4}
    \frac{\partial u_{j}(t)}{\partial t}=\epsilon\frac{u_{j-1}(t)-2 u_{j}(t)+ u_{j+1}(t)}{(\Delta x)^{2}}-\frac{1}{2}\left(u(t-\tau,x_{j})\cdot\frac{u(t-\tau,x_{j+1})-u(t-\tau,x_{j-1})}{\Delta x}\right)
\end{equation}
\[  +10x_{j}(1-x_{j})(1+x_{j}\sin(t x_{j})) \]
for $t\geq 0$ and $0\leq x\leq 1$, with the initial and boundary conditions
\[
\begin{cases}
    u_{0}(t)=u_{n}(t)=0, \\
    u_{j}(t)=\sin(\pi x_{j}) \quad t\leq 0.
\end{cases}
\]

 Denoting by $z_{j}^{i}$ the approximation of $u(t_{i},x_{j})$ and the vector $z(t)=(z_{1}(t),\cdots z_{n-1}(t))^{T}$, we get a new simplified  approximation of the form:
\begin{equation}\label{varBurgDDE}
    z'(t)=Az(t)+B(z(t-\tau))\cdot z(t-\tau)+\Tilde{f}(t,x),
\end{equation}
where the matrices used are defined as
\begin{equation*}
        A=\frac{\epsilon}{(\Delta x)^{2}}\begin{pmatrix}
            -2 & 1 & 0 & 0 & \cdots & 0 \\
            1 & -2 & 1 & 0 & \cdots & 0 \\
            0 & 1 & -2 & \ddots &  & \vdots \\
            \vdots &  & \ddots & \ddots & \ddots & 0\\
            \vdots &  &  & \ddots & -2 & 1\\
            0 & \cdots & \cdots & 0 & 1 & -2
        \end{pmatrix},
        \quad \quad
        B=\frac{1}{2 \Delta x}\begin{pmatrix}
            0 & -z^{i-m}_{1} & 0 & 0 & \cdots & 0 \\
            z^{i-m}_{2} & 0 & -z^{i-m}_{2} & 0 & \cdots & 0 \\
            0 & z^{i-m}_{3} & 0 & \ddots & & \vdots \\
            \vdots &  & \ddots & \ddots & \ddots & 0\\
            \vdots &  &  & \ddots & 0 & -z^{i-m}_{n-2}\\
            0 & \cdots & \cdots & 0 & z^{i-m}_{n-1} & 0
        \end{pmatrix},
\end{equation*}
and the \textit{j}th component of the vector $\Tilde{f}(t,x)$ is \ $10x_{j}(1-x_{j})(1+x_{j}\sin(t x_{j})).$     We use this equation (\ref{varBurgDDE}) only to analyze the stability of the IMEX BDF2 and IMEX BDF3 methods applied to (\ref{varBurg}). \\

However, there are still a few things we need to check before solving this DDE, that is whether the matrices A and B satisfy the conditions required to use our theory.    Firstly, we should notice that system (\ref{systemPr4}) is nonlinear, and our theory has been developed for linear systems.  Hence, the Lagrange's mean value theorem in several dimensions should be applied in a similar way as it was done in \cite{MV-Pagano} before we can use the theory provided above taking $B=B(z(\xi))$ as the Jacobian of the nonlinear part in an intermediate point $\xi \in [t_{i-m},t_{i-m+1}]$. In order not to further burden the analysis of nonlinear stability and the procedure proposed to carry it out, we used the information we have only on  the solution at $(t_{0-m}; z^{0-m})$, i.e.,  we will reduce to analyze $F(A^{-1}B(z_{-m}))$.  However, the complete analysis of the stability of similar methods for nonlinear PDDEs would require a more detail study.  In private conversations, Prof. Shirokoff suggested using a modification of a Lyapunov's theory.  But, as far as we know, this idea has not been developed in detail.

We can quickly see that the matrix A is hermitian since its conjugate transpose is the matrix itself. So as to know if it is also definite negative, let us calculate its eigenvalues. Thanks to the fact that $A$ is an Hermitian matrix, we can use for this the equation (\ref{eigenvaluesT}), imposing as in the previous example $\alpha=\frac{-2\epsilon}{(\Delta x)^{2}}$, $\beta=\gamma=\frac{\epsilon}{(\Delta x)^{2}}$ and $N=n-1$, so the eigenvalues of $A$ will be
$$\lambda_{r}=-\frac{2\epsilon}{k^{2}}\left(1+\cos{\left(\frac{r\pi}{n}\right)}\right),\quad r=1,...,n-1.$$
Knowing that both $\epsilon$ and $n$ are positive numbers by definition, we see that $\lambda_{r}<0 \ \forall r$ as before. In this regard, we conclude that the matrix A is hermitian and negative definite.

Now let us check if $A$ and $B$ diagonalize simultaneously. We know that this happens if and only if they commute ($A$ has all its eigenvalues simple) and, for this specific matrices, this statement is not fulfilled, so $A$ and $B$ do not diagonalize simultaneously.

\vspace{0.2cm}
 Finally, let us consider the IMEX BDF2 and IMEX BDF3 methods (\ref{IMEX-BDF2 general})-(\ref{IMEX-BDF3 general}), which applied to this particular nonlinear system have the form
\begin{equation*}
\frac{3}{2}z_{\tilde{n}+1}-2z_{\tilde{n}}+\frac{1}{2}z_{\tilde{n}-1}=h\bigg(A z_{\tilde{n}+1}+ (2 g( z_{\tilde{n}-m} )- g( z_{\tilde{n}-1-m}) )+\Tilde{f}(t_{\Tilde{n}+1},x)\bigg),
\end{equation*}
\begin{equation*}
\frac{11}{6}z_{\tilde{n}+1}-3z_{\tilde{n}}+\frac{3}{2}z_{\tilde{n}-1}-\frac{1}{3}z_{\tilde{n}-2}=h\bigg(Az_{\tilde{n}+1}+(3 g(z_{\tilde{n}-m})-3 g(z_{\tilde{n}-1-m})+g(z_{\tilde{n}-2-m} ) )+\Tilde{f}(t_{\Tilde{n}+1},x)\bigg),
\end{equation*}
being
\[   g(u)= \left(  -\frac{ u_1 u_2}{2 \Delta x } , -\frac{ u_2 (u_3-u_1)}{2 \Delta x }, \ldots,  -\frac{ u_{n-2} (u_{n-1}-u_{n-3})}{2 \Delta x },  \frac{ u_{n-1} u_{n-2}}{2 \Delta x }    \right). \]

\vspace{\baselineskip}
With all of this, we have discretized our partial delay differential equation and also have proven that $A$ and $B$ do not diagonalize simultaneously, so the way to go now in order to solve this example is by using Theorem \ref{Th: stability of the autonomous system and h}. Let $p=0$, $\epsilon=1$, $n=100$, $\tau=1$ and $t_{e}=20$ and let us see where $F(A^{-1}B(z_{-m}))$ lays in the complex plane for the matrices $A$ and $B(z_{ -m})$ defined above:
\begin{figure}[h]
        \centering
        \begin{subfigure}[b]{0.30\textwidth}
            \centering
            \includegraphics[width=\textwidth]{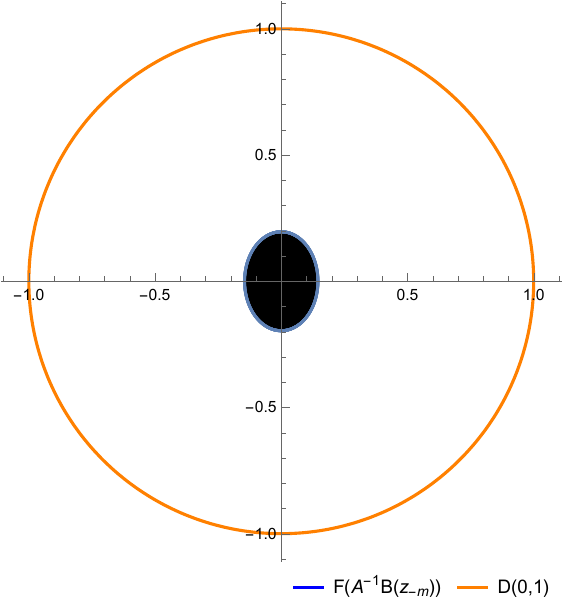}
             \caption{$F(A^{-1}B(z_{-m}))\subseteq D(0,1)$}
        \end{subfigure}
        \quad \quad
        \begin{subfigure}[b]{0.25\textwidth}
             \centering
             \includegraphics[width=\textwidth]{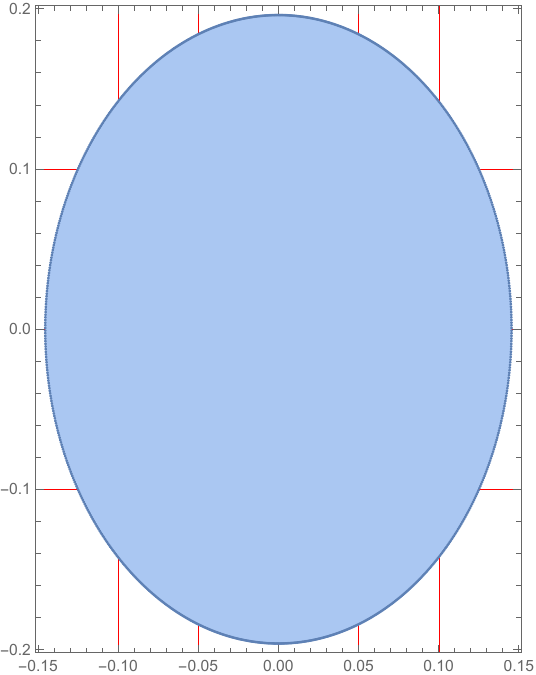}
             \caption{$F(A^{-1}B(z_{-m}))$}
             \label{F(AB)}
        \end{subfigure}
        \caption{Location of the numerical range $F(A^{-1}B(z_{-m}))$ for $n=100$.}
\end{figure}

\vspace{0.1cm}
It can be easily seen and calculated from Figure \ref{F(AB)} that $F(A^{-1}B(z_{-m}))\subseteq D(0,r)$ for $r=0.196$. Following the Theorem mentioned above, let us also consider the functions (\ref{chi-IMEX-BDF2}) and (\ref{chi-IMEX-BDF3}). We clearly see that $r$ is smaller than the minimum value of the first function, that is $r<\frac{1}{3}$, and therefore $r$ is inside the unconditionallly stability region. Thus, for any step size, the IMEX BDF2 method will be stable. For the second method, we can easily see that we are in the interval $\frac{1}{7}<r<0.218109$, so the step restriction will be given by the same inequality as in the previous example, leading for this specific values to $h^{*}=0.000448139$.

\vspace{\baselineskip}
Since the real solution of the problem is unkown, let us see in Figure \ref{fig:solVargBurg} the representation of the numerical solution of the problem obtained for the IMEX BDF2 and IMEX BDF3 methods, for a given step size. Note that, even for step sizes larger than $h^{*}$ in the IMEX BDF3, we get a good approximation to the real solution of the problem. This is due to the fact that both IMEX BDF2 and IMEX BDF3 have similar approximations and the method of second order mentioned is stable for any step size.
\begin{figure}[h]
    \centering
    \begin{subfigure}[b]{0.45\textwidth}
             \centering
             \includegraphics[width=\textwidth]{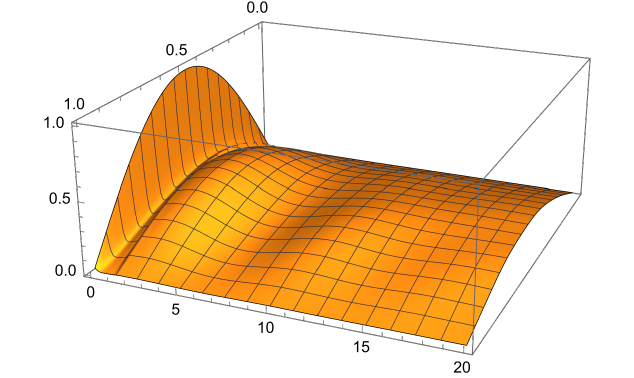}
             \caption{IMEX BDF2}
             \label{fig:solVargBurgIMEXBDF2}
    \end{subfigure}
    \quad \quad
    \begin{subfigure}[b]{0.40\textwidth}
             \centering
             \includegraphics[width=\textwidth]{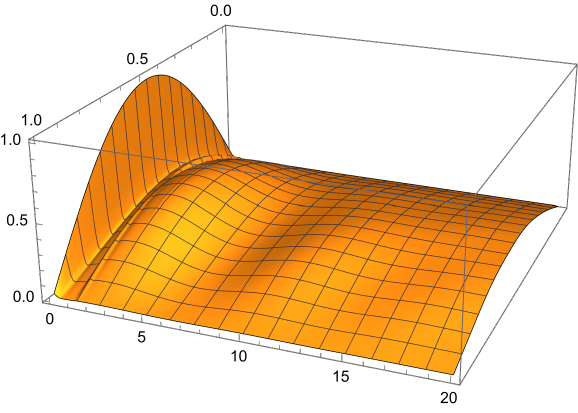}
             \caption{IMEX BDF3}
            \label{fig:solVargBurgIMEXBDF3}
    \end{subfigure}
    \caption{Numerical solution of (\ref{varBurg}) given by each method with $h=0.1$ and $t_{e}=20$.}\label{Test1}
    \label{fig:solVargBurg}
\end{figure}

\end{sol}

\section{Conclusions}\label{Conclusions}

 In this paper, the stability of IMEX-BDF methods for DDEs is studied in detail for the test equation $y'(t)=-A y(t) + B y(t-\tau)$, where $\tau$ is a constant delay, $A$ is a positive definite matrix, but $B$ might be any matrix.  Normally, the stability of numerical methods applied to DDEs is restricted to the scalar case, and therefore it allows the analysis to be extended to the case in which $A$ and $B$ diagonalize simultaneously. This work allows a step forward by considering the analysis in a much broader case, obtaining conditions both for obtaining unconditional stability, and for conditional stability depending on the step length.

 The numerical experiments, conducted on DDEs, and PDDEs of interest in applications, have highlighted the correctness of the theoretical analysis, and the good stability properties of the IMEX-BDF methods applied to DDEs.

 Furthermore, it is possible to use similar procedures to analyze the stability in similar multistep IMEX methods, and even of higher order. In this sense, we consider that this article can motivate future works.  Another topic that could be worthy of study in the future is the study of more precise sufficient conditions to have conditional stability when the eigenvalues of $A$ are very different.

\section*{Acknowledgments}

This research was funded by the Spanish Ministerio de Ciencia e Innovación (MCIN) with funding from the grant PID2023-148409NB-I00 MTM, Spanish Ministerio de Ciencia e Innovación (MCIN) with funding from the European Union NextGenerationEU (PRTRC17.I1); Consejer\'ia de Educación, Junta de Castilla y Le\'on, through QCAYLE project.  	
	The authors thank Prof. D. Shirokoff (co-author of \cite{LMM-Seib-Ros,LMM-Seib}), Alejandro Rodr\'iguez and Giovanni Pagano for fruitful conversations.

\bibliographystyle{abbrv}
\bibliography{refs}

\end{document}